\documentclass[10pt]{amsart}

\usepackage{amsmath,amssymb,amscd,amsthm,amsxtra}
\usepackage[dvips]{graphics,epsfig}

\allowdisplaybreaks[2]

\newtheorem{theorem}{Theorem} [section]
\newtheorem{maintheorem}{Theorem}
\newtheorem{lemma}[theorem]{Lemma}
\newtheorem{proposition}[theorem]{Proposition}
\newtheorem{remark}[theorem]{Remark} 

\newtheorem{definition}{Definition}
\newtheorem{corollary}[theorem]{Corollary}

\DeclareMathOperator*{\intt}{\int}

\DeclareMathOperator{\MAX}{MAX}

\newcommand{\noi}{\noindent}
\newcommand{\Z}{\mathbb{Z}}
\newcommand{\R}{\mathbb{R}}

\newcommand{\T}{\mathbb{T}}

\newcommand{\al}{\alpha}
\newcommand{\dl}{\delta}
\newcommand{\Dl}{\Delta}
\newcommand{\eps}{\varepsilon}

\newcommand{\g}{\gamma}
\newcommand{\G}{\Gamma}
\newcommand{\ld}{\lambda}

\newcommand{\s}{\sigma}
\newcommand{\ft}{\widehat}
\newcommand{\wt}{\widetilde}
\newcommand{\cj}{\overline}
\newcommand{\dx}{\partial_x}

\newcommand{\I}{\hspace{0.5mm}\text{I}\hspace{0.5mm}}
\newcommand{\II}{\text{I \hspace{-2.8mm} I} }
\newcommand{\III}{\text{I \hspace{-2.9mm} I \hspace{-2.9mm} I}}

\newcommand{\jb}[1]
{\langle #1 \rangle}

\begin{document}

\title
[ Invariance of Gibbs measure  for KdV Systems]
{\bf  Invariant Gibbs Measures and a.s. Global Well-Posedness for Coupled KdV Systems}

\author{Tadahiro Oh}

\address{Tadahiro Oh\\
Department of Mathematics\\
University of Toronto\\
40 St. George St, Rm 6290,
Toronto, ON M5S 2E4, Canada}

\email{oh@math.toronto.edu}




\begin{abstract}
We continue our study of the well-posedness theory of  a one-parameter family of  coupled KdV-type systems in the periodic setting.
When the value of a coupling parameter $\al \in (0, 4) \setminus \{1\}$, we show that
the Gibbs measure is invariant under the flow and 
the system is globally well-posed almost surely on the statistical ensemble,
provided that certain Diophantine conditions are satisfied.
\end{abstract}

\maketitle

\section{Introduction}

In this paper, we  consider coupled KdV systems of the form:
\begin{equation} \label{KDVsystem1}
\begin{cases}
u_t + a_{11} u_{xxx} + a_{12} v_{xxx} +  b_1 u u_x + b_2 u v_x + b_3 u_x v + b_4 v v_x = 0 \\ 
v_t + a_{21} u_{xxx} + a_{22} v_{xxx} + b_5 u u_x + b_6 u v_x + b_7 u_x v + b_8 v v_x  = 0 \\
(u, v) \big|_{t = 0} = (u_0, v_0)  
\end{cases}
\end{equation}

\noindent
in the periodic setting, where $A = \bigl(\begin{smallmatrix}
a_{11} & a_{12} \\ 
a_{21} & a_{22}
\end{smallmatrix} \bigr)$
is self-adjoint, and $u$ and $v$ are real-valued functions.
There are several systems of this type: the Gear-Grimshaw system \cite{GG}, the Hirota-Satsuma system \cite{HS}, 
the Majda-Biello system \cite{MB}, etc.
By applying the space-time scale changes along with the diagonalization of $A$, one can reduce \eqref{KDVsystem1} to 
\begin{equation} \label{KDVsystem2}
\begin{cases}
u_t +  u_{xxx} + \wt{b_1} u u_x + \wt{b_2} u v_x + \wt{b_3} u_x v + \wt{b_4} v v_x  = 0 \\ 
v_t + \al v_{xxx} + \wt{b_5} u u_x + \wt{b_6} u v_x + \wt{b_7} u_x v + \wt{b_8} v v_x  = 0 \\
(u, v) \big|_{t = 0} = (u_0, v_0),  
\end{cases}
\end{equation}

\noindent 
where $\al \ne 0$, 
$(x, t) \in \T\times\R$ with $\T = [0, 2\pi)$. 

In this paper, we assume that 
\eqref{KDVsystem2} has a Hamiltonian $H$ of the form 
``$H (u, v) = \frac{1}{2} \int u_x^2 + \al v_x^2+$nonlinear terms"
and that both $H$ and $N(u, v)  = \frac{1}{2} \int u^2 + b v^2$, for some $b > 0$, are conserved under the flow of \eqref{KDVsystem2}. 
(Note that this is the case for the Gear-Grimshaw and the Majda-Biello systems among other coupled KdV systems.)
When $\al \ne 1$, it is shown in \cite{OH1} that 
there is an interval $I_0$ around $\al = 1$ such that particular resonances occur for $\al \in I_0 \setminus \{1\}$
which are not present when $\al = 1$.
We show that,  for $\al \in I_0 \setminus \{1\}$, 
 the Gibbs measure $d \mu = Z^{-1} \exp(-\beta H(u, v) ) \prod_{x\in \T} d u(x) \otimes dv(x)$ is invariant under the flow \eqref{KDVsystem2}
and that \eqref{KDVsystem2} is globally well-posed almost surely on the statistical ensemble,
provided that certain Diophantine conditions are satisfied.

As a model example, we consider the Majda-Biello system:
\begin{equation} \label{MB}
\begin{cases}
u_t + u_{xxx} + vv_x = 0\\
v_t + \al v_{xxx} + (uv)_x = 0 \\
(u, v) \big|_{t = 0} = (u_0, v_0),  
\end{cases}
 \ 
(x, t) \in \T\times\R,
\end{equation}

\noindent
where $ \T = [0, 2\pi)$, $0< \alpha \leq 4$, and $u$ and $v$ are real-valued functions.
This system has been proposed by Majda and Biello \cite{MB} 
as a reduced asymptotic model to study the nonlinear resonant interactions of 
long wavelength equatorial Rossby waves and barotropic Rossby waves with a significant mid-latitude projection,
in the presence of suitable horizontally and vertically sheared zonal mean flows.
In \cite{MB},  the values of $\al $ are numerically determined and they are
$0.899$, $0.960$, and $0.980$ for different equatorial Rossby waves.
Of particular interest to us is the periodic case because of its challenging mathematical nature
as well as its physical relevance of the proposed model (the spatial period for the system before scaling is set as $40, 000$ km in \cite{MB}.)

Several conservation laws are known for the system:
\begin{equation} \label{Wconserved}
 E_1 = \int u \, dx, \ E_2 = \int v \, dx, \ N(u, v)  = \frac{1}{2}\int u^2 + v^2 dx, \ H(u, v)  =\frac{1}{2} \int u_x^2 + \al v_x^2 - u v^2 dx, 
\end{equation}

\noindent
where $H(u, v)$ is the Hamiltonian of the system.
There seems to be no other conservation law, suggesting that the Majda-Biello system may not be completely integrable.
The system has scaling which is similar to that of KdV and the critical Sobolev index $s_c$ is $-\frac{3}{2}$ just like KdV.

First, we review the local and global well-posedness (LWP and GWP) results of \eqref{MB} from \cite{OH1}, \cite{OH2}.
Note that all the results,  except for the global well-posedness on $\T$ for $\al \in (0, 1) \cup (1, 4)$,  are sharp in the sense that 
the smoothness/uniform continuity of the solution map fails below the specified regularities.
When $\al =1$, we showed in \cite{OH2} that \eqref{MB} is globally well-posed  
in $H^{-\frac{1}{2}}(\T) \times H^{-\frac{1}{2}}(\T)$
without the mean 0 condition on the initial data,
via the $I$-method developed by Colliander-Keel-Staffilani-Takaoka-Tao \cite{CKSTT4}
in the vector-valued variants $X^{s, b}_{p, q}$ of the Bourgain space $X^{s, b}$ \cite{BO1}.

Now, let's turn to the case $\al \in (0, 1) \cup (1, 4]$.
In this case, we have two distinct linear semigroups 
$S(t)= e^{ -t \dx^3}$ and $S_\al(t)= e^{ -\al t \dx^3}$ corresponding to the linear equations for $u$ and $v$.
Thus, we need to define two distinct Bourgain spaces $X^{s, b}$ and $X_\al^{s, b}$ to encompass the situation.
For $s, b \in \mathbb{R}$, let $X^{s, b}(\mathbb{T}\times\mathbb{R})$ and $X_\al^{s, b}(\mathbb{T}\times\mathbb{R})$ be the completion of 
the Schwartz class $\mathcal{S} (\mathbb{T} \times \mathbb{R})$ with respect to the norms
\begin{align} \label{WXsb1}
& \|u\|_{X^{s, b}(\mathbb{T} \times \mathbb{R})} = \big\| \jb{n}^s \jb{\tau -n^3}^b \ft{u}(n, \tau) \big\|_{L^2_{n, \tau}(\mathbb{Z} \times \mathbb{R})} 
\\ \label{WXsb2}
&  \|v\|_{X_{\al}^{s, b}(\mathbb{T} \times \mathbb{R})} = \big\| \jb{n}^s \jb{\tau - \al n^3}^b \ft{v}(n, \tau) \big\|_{L^2_{n, \tau}(\mathbb{Z} \times \mathbb{R})} ,
\end{align}

\noindent
where $\jb{\, \cdot \, } = 1 + |\cdot|$.
Then, two of the crucial bilinear estimates in establishing the LWP of \eqref{MB} are:
\begin{align} \label{bilinear1}
\| \dx (v_1 v_2)  \|_{X^{s, -\frac{1}{2}} (\mathbb{T} \times \mathbb{R} )} & \lesssim \|v_1\|_{X_\al^{s, \frac{1}{2}}(\mathbb{T} \times \mathbb{R} )}
\|v_2\|_{X_\al^{s, \frac{1}{2}}(\mathbb{T} \times \mathbb{R} ).} 
\\
\label{bilinear2}
\| \dx (u v)  \|_{X_\al^{s, -\frac{1}{2}} (\mathbb{T} \times \mathbb{R} )} & \lesssim \|u\|_{X^{s, \frac{1}{2}}(\mathbb{T} \times \mathbb{R} )}
\|v\|_{X_\al^{s, \frac{1}{2}}(\mathbb{T} \times \mathbb{R} ).} 
\end{align}

\noindent
First, consider the first bilinear estimate \eqref{bilinear1}.  
As in Kenig-Ponce-Vega  \cite{KPV4},  we define the bilinear operator $\mathcal{B}_{s, b} (\cdot, \cdot)$ by
\[ \mathcal{B}_{s, b} (f, g) (n, \tau) =  \frac{n \jb{n}^s }{\jb{\tau - n^3}^{\frac{1}{2}}}\frac{1}{2\pi}
\sum_{n_1+ n_2 = n} \intt_{\tau_1 + \tau_2 = \tau} \frac{f(n_1, \tau_1) g(n_2, \tau_2)}{\jb{n_1}^s 
\jb{n_2}^s \jb{\tau_1 - \al n_1^3}^\frac{1}{2}  \jb{\tau_2 - \al n_2^3}^\frac{1}{2}} d\tau_1.
\]
Then, \eqref{bilinear1} holds if and only if
$ \left\| \mathcal{B}_{s,b}(f, g) \right\|_{L^2_{n, \tau}} \lesssim \| f \|_{L^2_{n, \tau}} \| g \|_{L^2_{n, \tau}} . $
As in the KdV case,  $\dx$ appears on the left hand side of \eqref{bilinear1} and thus 
we need to make up for this loss of derivative from 
$\jb{\tau - n^3}^{\frac{1}{2}} \jb{\tau_1 - \al n_1^3}^\frac{1}{2}  \jb{\tau_2 - \al n_2^3}^\frac{1}{2}$ in the denominator.
Recall that we basically gain $\frac{3}{2}$ derivatives in the KdV case (with $n, n_1, n_2 \ne 0$) 
thanks to the algebraic identity
\begin{equation} \label{P1algebra}
n^3 - n_1^3 - n_2^3 = 3 n n_1n_2
\end{equation}

\noindent
for $n = n_1 + n_2$.
However, when $\al \ne 1$, we no longer have such an identity and 
we have
\begin{align} \label{resonance1}
\max \big( & \jb{\tau - n^3},    \jb{\tau_1 - \al n_1^3},   \jb{\tau_2 - \al n_2^3} \big)
\sim \jb{\tau - n^3}+  \jb{\tau_1 - \al n_1^3}+   \jb{\tau_2 - \al n_2^3} 
  \notag \\
& \gtrsim \big| (\tau - n^3) - (\tau_1 - \al n_1^3) - (\tau_2 - \al n_2^3) \big|  
 = | n^3 -  \al n_1^3 - \al n_2^3 | ,  
\end{align}
where $n = n_1 + n_2$ and  $\tau = \tau_1 + \tau_2$.
Note that the last expression in \eqref{resonance1} can be 0 for infinitely many (nonzero) values of $n, \: n_1,$ and $ n_2$,
causing resonances. 
By solving  the resonance equation:
\begin{equation} \label{JJreseq1}
n^3 - \al n_1^3 - \al n_2^3 = 0 \text{ with }  n = n_1 + n_2, 
\end{equation} 
we have $( n_1, n_2) = (c_1 n, c_2 n)$ or $(c_2 n, c_1 n)$, where 
\begin{equation} \label{c_1} 
c_1 = \tfrac{1}{2} + \tfrac{\sqrt{-3 + 12 \al^{-1}}}{6} \ \text{ and } \
c_2 = \tfrac{1}{2} - \tfrac{\sqrt{-3 + 12 \al^{-1}}}{6}.
\end{equation}
Note that $c_1 + c_2 = 1$ and that $c_1, c_2 \in \R$ if and only if $ 0 < \al \leq 4$. 
If $c_1 \in \mathbb{Q}$ (and thus $c_2 \in \mathbb{Q}$), then there are infinitely many values of $n \in \mathbb{Z}$ 
such that  $c_1n, \: c_2 n \in \mathbb{Z}$.  This causes resonances for infinitely many values of $n$, and 
thus we do not have any gain of derivative from $\jb{\tau - n^3} \jb{\tau_1 - \al n_1^3}  \jb{\tau_2 - \al n_2^3}  $ in this case.

If $c_1 \in \mathbb{R} \setminus \mathbb{Q}$,  then $c_1 n\notin \mathbb{Z}$ for any $n \in \mathbb{Z}$.
i.e. $n - \al n_1^3 - \al n_2^3 \ne 0 $ for any $n,  n_1, n_2 \in \mathbb{Z}$. 
However,  generally speaking, $n - \al n_1^3 - \al n_2^3$ can be  arbitrarily close to 0,
since $c_1 n$ can be arbitrarily close to an integer.
Therefore, we need  to measure {\it how ``close" $c_1$ is to rational numbers}.
In \cite{OH1}, we  used  the following definition regarding the Diophantine conditions commonly used in dynamical systems.

\begin{definition} [{Arnold \cite{AR}}]
A real number $\rho$ is called of type $(K, \nu)$ \textup{(}or simply of type ${\nu}$\textup{)} 
if there exist positive $K$ and $\nu$ such that for all pairs of integers $(m, n)$, we have
\begin{equation} \label{lowerbd}
\left| \rho - \frac{m}{n} \right| \geq \frac{K}{ |n|^{2+\nu}} .
\end{equation}
\end{definition}

\noindent
Also, for our purpose, we defined {\it the minimal type index} of a given real number $\rho$.
\begin{definition}
Given a real number $\rho$, define the minimal type index ${\nu_{\rho}}$ of ${\rho}$ by
\[ \nu_{\rho} = \begin{cases}
\infty \text{, if } \rho \in \mathbb{Q} \\
\inf \{ \nu > 0 : \rho \text{ is of type } \nu \} \text{, if } \rho \notin \mathbb{Q} .
\end{cases} \]
\end{definition}

\begin{remark} \label{JJDIO} \rm
Then, by Dirichlet Theorem {\cite[p.112]{AR}} and {\cite[p.116, lemma 3]{AR}},
it follows that $\nu_\rho \geq 0$ for   any $\rho \in \mathbb{R}$ and  $\nu_\rho = 0$ 
for almost every  $\rho \in \mathbb{R}$. 
\end{remark}

Using the minimal type index $\nu_{c_1}$ of $c_1$, for any $\eps >0$, we have
\begin{equation}\label{lowerbdC}
| n^3 -  \al n_1^3 - \al n_2^3 | \gtrsim |n|^{1-\nu_{c_1} -\eps}
\end{equation}
for all sufficiently large $n \in \mathbb{Z}$, which provides a good lower bound on \eqref{resonance1}.
With \eqref{lowerbdC}, we proved that \eqref{bilinear1} holds for $s > \frac{1}{2} + \frac{1}{2} \nu_{c_1}$.

The resonance equation of the second bilinear estimate \eqref{bilinear2}
 is given by
\begin{equation} \label{resonance2}
 \al n^3 - n_1^3 - \al n_2^3 = 0 \text{ with } n = n_1 + n_2.
\end{equation}

\noindent By solving \eqref{resonance2}, we obtain
$(n_1, n_2) = \big(d_1 n, (1-d_1) n\big),  \big(d_2 n, (1-d_2) n\big), (0, n)$, where 
\begin{equation} \label{d_1 and d_2}
d_1 = \tfrac{- 3 \al + \sqrt{3 \al(4 - \al)}}{2 (1 - \al)} \ 
\text{ and } \ d_2 = \tfrac{- 3 \al - \sqrt{3 \al(4 - \al)}}{2 (1 - \al)}.
\end{equation}
Note that $d_1, d_2 \in \mathbb{R}$ if and only if $\al \in [0, 1) \cup (1, 4]$.
Then, 
for any $\eps >0$, we have
\begin{equation}\label{lowerbdD}
|  \al n^3 - n_1^3 - \al n_2^3 | \gtrsim |n|^{1-\max( \nu_{d_1}, \nu_{d_2}) -\eps}
\end{equation}
for all sufficiently large $n \in \mathbb{Z}$.
With \eqref{lowerbdD}, we proved that 
\eqref{bilinear2} holds for $s > \frac{1}{2} + \frac{1}{2} \max( \nu_{d_1}, \nu_{d_2})$
with the mean 0 assumption on $u$.
Note that the mean 0 assumption on $u$ is needed 
since $n_1 = 0$ is a solution of \eqref{resonance2} for any $n \in \mathbb{Z}$.
Indeed, it is shown in \cite{OH1} that \eqref{bilinear2} fails for any $s \in \R$ without the mean 0 assumption on $u$.

\begin{remark} \label{REM:bilinear} \rm
We point out that the bilinear estimates \eqref{bilinear1} and \eqref{bilinear2} hold for $s \geq 0$ 
away from the resonance sets, i.e.
\eqref{bilinear1} holds for $s \geq 0$ on
$\{ (n, n_1) : |n| \gtrsim 1, |n_1 - c_1 n| \geq 1 \text{ and } |n_1 - c_2 n| \geq 1\}$,
and 
\eqref{bilinear2} holds for $s \geq 0$ on
$\{ (n, n_1) : |n| \gtrsim 1, |n_1 - d_1 n| \geq 1 \text{ and } |n_1 - d_2 n| \geq 1\}$.

\end{remark}

Now, let $s_0 (\al) = \frac{1}{2} + \frac{1}{2} \max( \nu_{c_1}, \nu_{d_1},  \nu_{d_2})$.
Note that $s_0 = \frac{1}{2}$ for almost every $\alpha \in (0, 4]$
in view of Remark \ref{JJDIO}.
In \cite{OH1}, we proved that, for $ \al \in (0, 4] \setminus \{1\}$, 
the Majda-Biello system \eqref{MB}  is locally well-posed in $H^s(\mathbb{T}) \times H^s(\mathbb{T})$
for  $s \geq s^\ast(\al) := \min ( 1,s_0 + ) $, 
assuming  the mean 0 condition on $u_0$.

We'd like to point out the following. 
On the one hand, we have  $s^\ast(\al) = s_0 (\al) =  \frac{1}{2}+$ 
for almost every $\al \in (0, 4] \setminus \{1\}$. 
On the other hand,   for any interval $I \subset (0, 4]$, there exists $\al \in I$ such that $s^\ast(\al) = 1$.
This shows that the well-posedness (below $H^1$) of the periodic Majda-Biello system  
is very unstable under a slight perturbation of the parameter $\al$.

Using the Hamiltonian $H(u, v)$, 
one can easily obtain the GWP of \eqref{MB} in $H^1(\T) \times H^1(\T)$.
This result is sharp when $s^\ast = 1$, 
i.e. when $\max( \nu_{c_1}, \nu_{d_1},  \nu_{d_2}) \geq 1$.
In particular, it is sharp for $\al = 4$, since $c_1 \in \mathbb{Q}$ for $\al = 4$.

When $\max( \nu_{c_1}, \nu_{d_1},  \nu_{d_2}) < 1$, we used the $I$-method
to generate a sequence of modified energies (modified Hamiltonians) $H^{(j)}(u, v) (t)  $.
Using the second modified energy $H^{(2)}$, 
it is shown in \cite{OH2} that 
\eqref{MB} is globally well-posed in $H^s(\T) \times H^s(\T)$
for 
$ s \geq s^{**} := \max \Big( \frac{6 (s_0+) - 2(s_0+)^2}{5 - (s_0+)}, \frac{2(s_0+) + 9}{14}\Big)$.
In particular, we have $s^{**} (\al) = \frac{5}{7} + $ for almost every $\al \in (0, 4] \setminus \{1\}$.
We also showed that the Fourier multiplier for the time derivative of the third modified energy $H^{(3)}$ is unbounded on a nontrivial set.
i.e. the $I$-method fails before the GWP result matches the LWP one.
This shows that the GWP obtained via $H^{(2)}$ is the best possible result using the $I$-method.
Note that there is a gap between $s^*(\al) $ and $s^{**}(\al)$, unless $s^* = 1$.
In particular, $s^* = \frac{1}{2} + > s^{**} = \frac{5}{7}+$ for a.e. $\al \in (0, 4]\setminus \{1\}$.

In an attempt to fill the gap between the LWP and the GWP results, 
we consider the Gibbs measure of the form ``$d\mu = Z^{-1}\exp(-\beta H(u, v) ) \prod_{x\in \T} d u(x) \otimes dv(x)$".
First recall the following; Given a Hamiltonian flow
\begin{equation*}
\begin{cases}
\dot{p}_i = \frac{\partial H}{\partial q_i} \\
\dot{q}_i = - \frac{\partial H}{\partial p_i} 
\end{cases}
\end{equation*}
on $\mathbb{R}^{2n}$ with Hamiltonian $ H = H(p_1, \cdots, p_n, q_1, \cdots, q_n)$,
Liouville's theorem states that the Lebesgue measure on $\mathbb{R}^{2n}$ is invariant under the flow.
From the conservation of the Hamiltonian $H$, the Gibbs measures
$e^{-\beta H} \prod_{i = 1}^{n} dp_i dq_i$ are also invariant,
where $\beta$ is the reciprocal temperature. 
In our context, the Hamiltonian $H(u, v)$ is conserved under the flow of \eqref{MB}.
Then, we'd like to use the invariance of the Gibbs measure $\mu$ 
(which holds true in finite dimensional case)
to prove a GWP result.
At this point, everything is merely formal, which needs to be made rigorous.

In the context of NLS, Lebowitz-Rose-Speer \cite{LRS} considered the Gibbs measure of the form 
$d \mu = \exp (-\beta H(u)) \prod_{x\in \T} d u(x)$ where 
$H(u)$ is the Hamiltonian given by $H(u) = \frac{1}{2} \int |u_x|^2 \pm \frac{1}{p} \int |u|^p dx$.
In the focusing case (with  $-$), $H(u)$ is not bounded from below and this causes a problem.
Using the conservation of the $L^2$ norm,
they instead considered the Gibbs measure of the form 
$d \mu = \exp (-\beta H(u)) \chi_{\{\|u\|_{L^2} \leq B \}} \prod_{x\in \T} d u(x)$, i.e. with an $L^2$-cutoff.
This turned out to be a well-defined measure on $H^{\frac{1}{2}-}(\T) = \bigcap_{s<\frac{1}{2}} H^s(\T)$
(for $p < 6$ with any $B>0$, and $ p = 6$ with sufficiently small $B$.)
Bourgain \cite{BO4} continued this study and 
proved the invariance of $\mu$ under the flow of NLS
and the global well-posedness almost surely on the statistical ensemble.
Note that  \cite{BO4} appeared before the so-called Bourgain's method \cite{BO2} or the $I$-method \cite{CKSTT4},
i.e. there was virtually no method available to establish any GWP result 
from a LWP result whose regularity was between two conservation laws.
This was the case for NLS for $  4< p \leq 6$.
We use this idea to obtain a.s. GWP of the Majda-Biello system \eqref{MB}.
Recently, Burq-Tzvetkov \cite{BT1} independently and simultaneously used 
similar ideas to prove a.s. GWP for 
the nonlinear wave equation on the unit ball in $\mathbb{R}^3$ under the radial symmetry.
Also, see other work by Tzvetkov related to this subject \cite{TZ1}, \cite{TZ2}.

For the Majda-Biello system \eqref{MB}, we have $H(u, v) = \frac{1}{2} \int u_x^2 + \al v_x^2 - u v^2 dx$. 
It is known (c.f. Zhidokov \cite{Z}) that the Wiener measure 
$d\rho = \wt{Z}^{-1} \exp( -\frac{1}{2}\int u_x^2 + \al v_x^2dx) \prod_{x \in \T} d u(x) \otimes dv(x)$
is a well-defined countably additive measure on $H^{\frac{1}{2}-} (\T) \times H^{\frac{1}{2}-} (\T)$.
Since $-\frac{1}{2}\int u v^2 dx$ is not sign-definite, 
we need to add an $L^2$ cutoff in considering the Gibbs measure (weighted Wiener measure) as in \cite{LRS} and \cite{BO4}.
Then, $d\mu = Z^{-1} \exp(\frac{1}{2} \int u v^2 dx) \chi_{\{\| (u, v) \|_{L^2} \leq B \}}d\rho$
is a well-defined countably additive measure on $H^{\frac{1}{2}-} (\T) \times H^{\frac{1}{2}-} (\T)$.
When $\al = 1$, the invariance of the Gibbs measure $\mu$ directly follows
from Bourgain's argument for KdV in \cite{BO4} and the GWP of \eqref{MB} in 
$H^{-\frac{1}{2}} (\T) \times H^{-\frac{1}{2}} (\T)$ obtained in \cite{OH2}.

Now, recall that for $\al \in (0, 4) \setminus\{1\}$, 
the Majda-Biello system \eqref{MB} is LWP in $H^s(\T) \times H^s(\T)$
only for $s > \frac{1}{2} + \frac{1}{2}\max (\nu_{c_1},\nu_{d_1},\nu_{d_2} ) \geq \frac{1}{2}$
and is $C^3$ ill-posed for $s < \frac{1}{2} + \frac{1}{2}\max (\nu_{c_1},\nu_{d_1},\nu_{d_2} )$ 
in the sense that the solution map is not $C^3$  (c.f. \cite{OH1}.)
i.e. the flow of \eqref{MB} is not well-defined on the support of the Gibbs measure $\mu$
in terms of the usual Sobolev spaces.
We instead consider the Cauchy problem for $(u_0, v_0) \in H^{s_1, s_2}(\T) \times H^{s_1, s_2}(\T)$, 
where 
\begin{equation} \label{GibbsIV} 
\| \phi \|_{H^{s_1, s_2}} = \| \phi \|_{H^{s_1}} + \sup_{n} \jb{n}^{s_2} |\ft{\phi}(n)| < \infty 
\end{equation}
\noindent
for some $s_1, s_2$ with $0 < s_1 < \frac{1}{2} < s_2 < 1$ 
(to be determined later.)
First, recall that, as mentioned in Remark \ref{REM:bilinear},  
the bilinear estimates \eqref{bilinear1} and \eqref{bilinear2} 
fail  for $s \leq \frac{1}{2}$ only at 4 frequencies near the resonances.
$\sup_{n} \jb{n}^{s_2} |\ft{\phi}(n)|$ in \eqref{GibbsIV} exactly controls these particular resonances with the higher regularity 
$s_2 > \frac{1}{2} + \frac{1}{2}\max( \nu_{c_1},\nu_{d_1}, \nu_{d_2})$.
We have the following theorem.

\begin{maintheorem} \label{THM:LWP1}
Let $\al \in (0, 4) \setminus\{1\}$ and $\max( \nu_{c_1},\nu_{d_1}, \nu_{d_2}) < 1$. Assume the mean 0 condition on $u_0$.
Let $\frac{1}{4} < s_1 < \frac{1}{2} < s_2 < 1 $ with $2s_1 > s_2 > \frac{1}{2} + \frac{1}{2}\max( \nu_{c_1},\nu_{d_1}, \nu_{d_2}) $.
Then, the Majda-Biello system \eqref{MB} is locally well-posed in $H^{s_1, s_2}(\T) \times H^{s_1, s_2}(\T)$.
\end{maintheorem}

As seen in Bourgain's work on mKdV and Zakharov system \cite{BO4}, \cite{BO5}, 
 we have $\mu (H^{s_1, s_2} \times H^{s_1, s_2}) = 1$ for $0 < s_1 < \frac{1}{2} < s_2 < 1 $.
i.e. $H^{s_1, s_2}(\T) \times H^{s_1, s_2}(\T)$ contains the full support of $\mu$.
In \cite{BO4}, Bourgain proved the invariance of the Gibbs measure and a.s. GWP of mKdV
by establishing an improved local well-posedness in $H^{s_1, s_2}$ with $s_1 = \frac{1}{2}-$
and $s_2 =1- $.
Following his argument, 
we obtain the a.s. global well-posedness of \eqref{MB},
using the finite dimensional approximation to \eqref{MB} 
along with the invariance of the finite dimensional Gibbs measure.

\begin{maintheorem} \label{THM:GWP1}
Let $\al \in (0, 4) \setminus \{1\}$ and $\max( \nu_{c_1},\nu_{d_1}, \nu_{d_2}) < 1$. Assume the mean 0 condition on $u_0$.
Then, the Gibbs measure $\mu$ (with an $L^2$-cutoff) is invariant under the flow of \eqref{MB}, 
and \eqref{MB} is globally well-posed a.s. on the statistical ensemble.
\end{maintheorem}

We point out that Theorem \ref{THM:GWP1} does not fill the gap between the LWP and GWP of \eqref{MB}
as we initially hoped.
However,  it does establish a new GWP result 
for almost every $\al \in (0, 4) \setminus\{1\}$ which can not be obtained by the methods in \cite{OH1} and \cite{OH2}.

This work is a  part of the author's Ph.D. thesis \cite{OHTHESIS}.
This paper is organized as follows:
In Section 2, we introduce some standard notations.
In Section 3, we go over the basic theory of Gaussian Hilbert spaces and abstract Wiener spaces, and 
we give the precise meaning to the Gibbs measure $\mu$.
In Section 4, we introduce the function spaces and linear estimates.
Then,  we prove Theorem \ref{THM:LWP1} in Section 5, 
and extend this local result to a.s. GWP in Section 6.
We also establish the invariance of the Gibbs measure $\mu$. 
In Appendix, we present the proof of a probabilistic lemma from Section 3.

\smallskip

\noindent
{\bf Acknowledgements:} 
The author would like to express his sincere gratitude to his Ph.D. advisor, Prof. Andrea R. Nahmod.
He acknowledges the NSF summer support  in 2005--06 under Prof. Nahmod's grant DMS 0503542.
He is also grateful to Prof. Luc Rey-Bellet for helpful discussions in probability
and abstract Wiener spaces.

\section{Notation}

On $\T$, the spatial Fourier domain is $\Z$.
Let $dn$ be the normalized counting measure on $\Z$, 
and we say $f \in L^p(\Z)$, $1 \leq p < \infty$, if
\[ \| f \|_{L^p(\mathbb{Z})} = \bigg( \int_{\mathbb{Z}} |f(n)|^p dn \bigg)^\frac{1}{p}  
:= \bigg( \frac{1}{2\pi} \sum_{n \in \mathbb{Z}} |f(n)|^p \bigg)^\frac{1}{p} < \infty.\]

\noindent
If $ p = \infty$, we have the obvious definition involving the essential supremum.
We often drop $2\pi$ for simplicity.
If the function depends on both $x$ and $t$, we use ${}^{\wedge_x}$ 
(and ${}^{\wedge_t}$) to denote the spatial (and temporal) Fourier transform, respectively.
However, when there is no confusion, we simply use ${}^\wedge$ to denote the spatial Fourier transform,
temporal Fourier transform, and  the space-time Fourier transform, depending on the context.

Let $X^{s, b}$ and $X_\al^{s, b}$ be as in \eqref{WXsb1} and \eqref{WXsb2}.
Given any time interval $I = [t_1, t_2]\subset \mathbb{R}$, we define the local in time $X^{s, b}(\T \times I )$ 
(or simply $X^{s, b}[t_1, t_2]$) by
\[ \|u\|_{X_I^{s, b}} 
= \|u \|_{X^{s, b}(\T \times I )} = \inf \big\{ \|\wt{u} \|_{X^{s, b}(\T \times \mathbb{R})}: {\wt{u}|_I = u}\big\}.\]

\noindent
We define the local in time $X^{s, b}_\al (\T \times I )$ analogously.
Also, in dealing with a product space of two copies of a Banach space $X$, 
we may use $X\times X$ and $X$ interchangeably.

Let $\eta \in C^\infty_c(\mathbb{R})$ be a smooth cutoff function supported on $[-2, 2]$ with $\eta \equiv 1$ on $[-1, 1]$
and let $\eta_{_T}(t) =\eta(T^{-1}t)$. 
We use $c,$ $ C$ to denote various constants, usually depending only on $s_1, s_2, b$, and  $\al$.
If a constant depends on other quantities, we will make it explicit.
We use $A\lesssim B$ to denote an estimate of the form $A\leq CB$.
Similarly, we use $A\sim B$ to denote $A\lesssim B$ and $B\lesssim A$
and use $A\ll B$ when there is no general constant $C$ such that $B \leq CA$.
We also use $a+$ (and $a-$) to denote $a + \eps$ (and $a - \eps$), respectively,  
for arbitrarily small $\eps \ll 1$.

\section{Gaussian Measure in Hilbert Space and Abstract Wiener Space}

In this section, we go over the basic theory of Gaussian measures in Hilbert spaces
and abstract Wiener spaces
and provide the precise meaning of the Gibbs measure 
``$d \mu = Z^{-1} \exp(-\beta H(u, v) ) \prod_{x\in \T} d u(x) \otimes dv(x)$"
appearing in Section 1.
For simplicity, we set the reciprocal temperature $\beta = 1$.
For details, see  Zhidokov \cite{Z}, Gross \cite{GROSS}, and Kuo \cite{HUO}.

First, recall  (centered) Gaussian measures in $\mathbb{R}^n$.
Let $n \in \mathbb{N}$ and $B$ be a symmetric positive $n \times n$ matrix
with real entries.
The  Borel measure $\rho$ in $\mathbb{R}^n$ with the density
\[ d \rho(x) = \frac{1}{\sqrt{(2\pi)^n \det (B )}} \exp \big( -\tfrac{1}{2} \langle B^{-1} x, x \rangle_{\mathbb{R}^n} \big)\]

\noindent
is called a (nondegenerate centered ) Gaussian measure in $\mathbb{R}^n$.
Note that $\rho(\mathbb{R}^n) = 1$.

Now, we consider the analogous definition for the infinite dimensional (centered) Gaussian measures.
Let $H$ be a real separable Hilbert space and $B: H \to H$ be a linear positive self-adjoint operator 
(generally not bounded) with eigenvalues $\{\ld_n\}_{n\in \mathbb{N}}$
and the corresponding eigenvectors $\{e_n\}_{n\in\mathbb{N}}$  forming an orthonormal basis of $H$.
We call a set $M \subset H$  cylindrical if there exists an integer $n\geq 1$ and a Borel set $F \subset \mathbb{R}^n$
such that
\begin{equation} \label{CYLINDER}
 M = \big\{ x \in H : ( \jb{ x, e_1}_H, \cdots, \jb{ x, e_n}_H ) \in F \big\}. 
\end{equation}

\noindent
For a fixed operator $B$ as above, we denote by $\mathcal{A}$ the set of all cylindrical subsets of $H$.
Note that $\mathcal{A}$ is a field.
Then, the centered Gaussian measure in $H$ with the correlation operator $B$ is defined as 
the additive (but not countably additive in general) measure $\rho$ defined on the field $\mathcal{A}$
via
\begin{equation} \label{CGAUSSIAN}
 \rho(M) = (2\pi)^{-\frac{n}{2}} \prod_{j = 1}^n \ld_j^{-\frac{1}{2}} \int_F e^{-\frac{1}{2}\sum_{j = 1}^n \ld_j^{-1} x_j^2 }d x_1 \cdots dx_n,
\text{ for }M \in \mathcal{A} \text{ as in \eqref{CYLINDER}. }
\end{equation}

\noindent
The following theorem tells us when  this Gaussian measure $\rho$ is countably additive.

\begin{theorem} \label{COUNTABLEADD}
The  Gaussian measure $\rho$ defined in \eqref{CGAUSSIAN} is countably additive
on the field $\mathcal{A}$ if and only if $B$ is an operator of trace class, 
i.e. $\sum_{n = 1}^\infty \ld_n < \infty$.
If the latter holds, then
the minimal $\s$-field $\mathcal{M}$ containing the field $\mathcal{A}$ of all cylindrical sets is the Borel $\s$-field on $H$.
\end{theorem}

Consider a sequence of the finite dimensional Gaussian measures $\{\rho_n\}_{n\in\mathbb{N}}$ as follows.
For fixed $n \in \mathbb{N}$, let $\mathcal{M}_n$ be the set of all cylindrical sets in $H$ of the form \eqref{CYLINDER} with this fixed $n$
and arbitrary Borel sets $F\subset \mathbb{R}^n$.
Clearly, $\mathcal{M}_n$ is a $\s$-field, and setting 
\[ \rho_n(M) = (2\pi)^{-\frac{n}{2}} \prod_{j = 1}^n \ld_j^{-\frac{1}{2}} \int_F e^{-\frac{1}{2}\sum_{j = 1}^n \ld_j^{-1} x_j^2 }d x_1 \cdots dx_n\]

\noindent
for $M \in \mathcal{M}_n$, we obtain a countably additive measure $\rho_n$ defined on $\mathcal{M}_n$.
Then, one can show that each measure $\rho_n$ can be naturally extended onto the whole Borel $\s$-field $\mathcal{M}$ of $H$
by $ \rho_n(A) := \rho_n(A \cap \text{span}\{e_1, \cdots, e_n\})$
for $A \in \mathcal{M}$.
Then, we have

\begin{proposition} \label{PROP:Zhidkov2}
Let $\rho$ in  \eqref{CGAUSSIAN} be countably additive.
Then,  $\{\rho_n\}_{n\in \mathbb{N}}$ constructed above converges weakly to $\rho$ as $n \to \infty$.

\end{proposition}

Now, we construct the Gaussian measure which comes from the quadratic part of the Hamiltonian
$H(\phi, \psi) = \frac{1}{2} \int \phi^2_x + \al \psi^2_x - \phi \psi^2 dx$, $\al > 0$.
First, we identify a pair of real-valued functions $(\phi, \psi)$ on $\mathbb{T}$ with their Fourier coefficients $(a, b) = (a_n, b_n)_{n\in \mathbb{Z}}$.
Since $\phi $ and $\psi$ are real-valued, we have $a_{-n} = \cj{a_n}$ and $b_{-n} = \cj{b_n}$. 
Then, define the finite dimensional Gaussian measure $\rho_N$ on $\mathbb{C}^{N} \times \mathbb{C}^{N} 
= \big\{ (a_n, b_n) : 1 \leq n \leq N \big\}$
with the normalized density
\begin{equation}\label{WienerN}
 d \rho_N = \wt{Z}_N^{-1} e^{-\frac{1}{2} \sum_{n = 1}^N n^2 (|a_n|^2+ \al |b_n|^2)} \textstyle \prod_{ n = 1 }^N d ( a_n \otimes b_n), 
\end{equation}

\noindent
where
$\wt{Z}_N = \int_{\mathbb{C}^{N} \times \mathbb{C}^{N}} 
e^{-\frac{1}{2} \sum_{ n = 1 }^N n^2 (|a_n|^2+ \al |b_n|^2)} \prod_{ n = 1 }^N  d ( a_n \otimes b_n). $
Note that this measure is the induced probability measure on $\mathbb{C}^{N} \times \mathbb{C}^{N}$ under the map
$ \omega \mapsto \big\{ \big(n^{-1} {f_n(\omega)},  {\al^{-\frac{1}{2}}n^{-1}}{g_n(\omega)} \big): 1 \leq n \leq N \big\},$
\noindent
where $\{f_n(\omega)\}$ and $\{g_n(\omega)\}$ are i.i.d. standard complex Gaussian random variables.
In particular,  
 $\rho_N$ is a Wiener measure on $\mathbb{C}^{2N}$.
Next, define 
\begin{equation} \label{Wiener}
 d \rho = \wt{Z}^{-1} e^{-\frac{1}{2} \sum_{  n \geq 1 } n^2 (|a_n|^2+ \al |b_n|^2)} \textstyle \prod_{  n \geq 1 } d ( a_n \otimes b_n), 
\end{equation}
where
$\wt{Z} = \int e^{-\frac{1}{2} \sum_{   n \geq 1 } n^2 (|a_n|^2+ \al |b_n|^2)}  \prod_{  n \geq 1 } d ( a_n \otimes b_n). $
For now, assume the means of $\phi$ and $\psi$ on $\T$ are 0, i.e. $a_0 = b_0 = 0$.
Let $\dot{H}^s_0$ be the homogeneous Sobolev space restricted to the real-valued mean 0 elements.
Then, we'd like to know for which $s \in \R$  the Gaussian measure $\rho$ with the density
can be a well-defined countably additive measure on $\dot{H}^s_0 \times \dot{H}^s_0$.

For simplicity, we consider a Gaussian on a space of a single real-valued function.
Let $\jb{\cdot, \cdot}_{\dot{H}^s}$ be the usual inner product in $\dot{H}^s$.
i.e. $ \big\langle \sum c_n e^{inx}, \sum d_n e^{inx} \big\rangle_{\dot{H}^s} = \sum_{n \ne 0} |n|^{2s} c_n \cj{d_n} $.
Let $B_s = \sqrt{-\Dl}\vphantom{|}^{2s-2}$.
Then, the weighted exponentials $\{|n|^{-s} e^{inx}\}_{n\ne 0}$ are the eigenvectors of $B_s$ with the eigenvalue $|n|^{2s-2}$,
forming an orthonormal basis of $H^s_0$.
Note that
\[ -\tfrac{1}{2}\jb{B^{-1} \phi, \phi}_{\dot{H^s}} = -\tfrac{1}{2}\Big\langle \sum_{n \ne 0} |n|^{2-2s} a_n e^{inx}, \sum_{n \ne 0} a_n e^{inx} \Big\rangle_{\dot{H}^s} 
= -\tfrac{1}{2}\sum_{n \ne 0} |n|^{2} |a_n|^2.\] 

\noindent
The right hand side is exactly the expression appearing in the exponent in \eqref{Wiener}.
By Theorem \ref{COUNTABLEADD}, $\rho$ is  countably additive
if and only if $B$ is  of trace class, i.e.
$ \sum_{n \ne 0} |n|^{2s-2} < \infty$ if and only if $ s < \frac{1}{2}.$
Hence, $\bigcap_{s < \frac{1}{2}} H^s \times H^s$ is a natural space to work on.

Unfortunately, it is shown in \cite{OH1} that 
\eqref{MB} is ill-posed in $H^s \times H^s$ for $s < \frac{1}{2}$  
when $\al \in (0, 4) \setminus\{1\}$.
In view of Theorems \ref{THM:LWP1}, 
we consider  the property of $\rho$ on 
$H^{s_1, s_2} \times H^{s_1, s_2}$ for $0 < s_1 < \frac{1}{2} < s_2 < 1$.
Since $H^{s_1, s_2} \times H^{s_1, s_2}$ is not a Hilbert space, we now turn to the basic theory of abstract Wiener spaces.

Recall the following definitions \cite{HUO}:
Given  a real separable Hilbert space $H$ with norm $\|\cdot \|$, 
let $\mathcal{F} $ denote the set of finite dimensional orthogonal projections $\mathbb{P}$ of $H$.
Then, define a cylinder set $E$ by  $E = \{ x \in H: \mathbb{P}x \in F\}$ where $\mathbb{P} \in \mathcal{F}$ 
and $F$ is a Borel subset of $\mathbb{P}H$,
and let $\mathcal{R} $ denote the collection of such cylinder sets.
Note that $\mathcal{R}$ is a field but not a $\s$-field.
Then, the Gauss measure $\rho$ on $H$ is defined 
by 
\[ \rho(E) = (2\pi)^{-\frac{n}{2}} \int_F e^{-\frac{\|x\|^2}{2}} dx  \]

\noindent
for $E \in \mathcal{R}$, where
$n = \text{dim} \mathbb{P} H$ and  
$dx$ is the Lebesgue measure on $\mathbb{P}H$.
It is known that $\rho$ is finitely additive but not countably additive in $\mathcal{R}$.

A seminorm $|||\cdot|||$ in $H$ is called measurable if for every $\eps>0$, 
there exists $\mathbb{P}_0 \in \mathcal{F}$ such that 
$ \rho( ||| \mathbb{P} x ||| > \eps  )< \eps$
\noindent
for $\mathbb{P} \in \mathcal{F}$ orthogonal to $\mathbb{P}_0$.
Any measurable seminorm  is weaker  than the norm of $H$,
and $H$ is not complete with respect to $|||\cdot|||$ unless $H$ is finite dimensional.
Let $B$ be the completion of $H$ with respect to $|||\cdot|||$
and denote by $i$ the inclusion map of $H$ into $B$.
The triple $(i, H, B)$ is called an abstract Wiener space.

Now, regarding $y \in B^\ast$ as an element of $H^\ast \equiv H$ by restriction,
we embed $B^\ast $ in $H$.
Define, for a Borel set $F \subset \R^n$, 
\[ \wt{\rho} ( \{x \in B: ((x, y_1), \cdots, (x, y_n) )\in F\})
= \rho ( \{x \in H: (\jb{x, y_1}_H, \cdots, \jb{x, y_n}_H )\in F\}),\]

\noindent
where $y_j$'s are in $B^\ast$ and $(\cdot , \cdot )$ denote the natural pairing between $B$ and $B^\ast$.
Let $\mathcal{R}_B$ denote the collection of cylinder sets
$ \{x \in B: ((x, y_1), \cdots, (x, y_n) )\in F \}$
in $B$.

\begin{theorem}[Gross \cite{GROSS}]
$\wt{\rho}$ is countably additive in the $\s$-field generated by $\mathcal{R}_B$.
\end{theorem}

In the present context, let $H = H^1_0 \times H^1_0$ and   
$B = H^{s_1, s_2} \times H^{s_1, s_2}$ with $0 < s_1 < \frac{1}{2} < s_2 < 1$. 
Then, it basically follows from (the proof of) Lemma  \ref{LEM:tight}
that the seminorms  $\|\cdot\|_{B}$ is measurable.
Hence, $(i, H, B)$ is an abstract Wiener space, 
and $\rho$ in \eqref{Wiener} is countably additive in $B$.

Next, we consider the full Gibbs measure ``$d \mu = Z^{-1} \exp(- H(\phi, \psi) ) \prod_{x\in \T} d \phi(x) \otimes d\psi(x)$"
where $H(\phi, \psi) =\frac{1}{2} \int \phi_x^2 +  \al \psi_x^2 - \phi \psi^2 dx$.
As for the KdV case, $\int \phi \psi^2 dx$ is not sign-definite
and thus we restrict ourselves to the ball of radius $B>0$ in $L^2 \times L^2$. See \cite{LRS}, \cite{BO4}.

Let $ \Omega_N = \{(a_n, b_n) :{ 0 \leq n \leq N } \}$
and $\Omega = \{(a_n, b_n) :{ n \geq 0} \}$.
Let $B$ be a cutoff on the $L^2$  norm and consider the ball in $\mathbb{C}^{N+1} \times \mathbb{C}^{N+1}$ given by
\[ \Omega_{N, B} = \big\{  (a_n, b_n)_{ 0 \leq n \leq N } : \| (a_n, b_n) \|_{L^2_n}  \leq B \big\} .\]

\noindent
(Recall $a_{-n} = \cj{a_n}$ and $b_{-n} = \cj{b_n}$.) 
Also, define 
$ \Omega_B = \left\{  (a_n, b_n)_{ n \geq 0} : \| (a_n, b_n) \|_{L^2_n}  \leq B \right\}. $
Let $\mathbb{P}_N$ be the projection onto the Fourier modes $\leq N $ given by
$ \mathbb{P}_N \phi =\phi^N =  \sum_{|n| \leq N} a_n e^{ i nx}. $
Then, we have
$\rho \circ \mathbb{P}_N^{-1}  = \rho_N$.
Now, define the weighted Wiener measure $\mu_N$ on $\mathbb{C}^{N+1} \times \mathbb{C}^{N+ 1} 
= \big\{ (a_n, b_n):{0 \leq n \leq N } \big\}$ 
 by
\begin{equation} \label{GmuN}
 d \mu_N = Z_N^{-1} \exp \bigg( \frac{1}{2} \int \mathbb{P}_N \phi (\mathbb{P}_N \psi )^2dx \bigg) \chi_{\Omega_{N, B}} \ d (a_0,  b_0) \otimes d \rho_N,  
\end{equation}

\noindent
where
$ Z_N = \int_{\mathbb{C}^{N+1} \times \mathbb{C}^{N+ 1} }
\exp \big( \frac{1}{2} \int \mathbb{P}_N \phi (\mathbb{P}_N \psi \big)^2dx \big) \chi_{\Omega_{N, B}} \ d  (a_0,  b_0) \otimes d \rho_N $,
and $d a_0$ and $db_0$ are the Lebesgue measures on $\mathbb{C}$.
Similarly,  the weighted Wiener measure $\mu$ on $\big\{ (a_n, b_n):{  n \geq 0 } \big\}$ by
\begin{equation} \label{Gmu}
 d \mu = Z^{-1} \exp \left( \frac{1}{2} \int \phi  \, \psi^2dx \right) \chi_{\Omega_B} \ d  (a_0,  b_0) \otimes d \rho, 
\end{equation}

\noindent
where
$Z = \int \exp \left( \frac{1}{2} \int \phi  \, \psi^2dx \right) \chi_{\Omega_B} \ d  (a_0,  b_0) \otimes d \rho. $
At this point, $Z$ need not be finite.
Indeed, the result below follows from \cite{LRS} and   \cite{BO4}.

\begin{lemma} \label{LEM:GABSCONTI}
For any $ r < \infty$, we have
\begin{align} \label{GABSCONTI1}
  \exp \left( \frac{1}{2} \int \mathbb{P}_N \phi (\mathbb{P}_N \psi )^2dx \right) \chi_{\Omega_{N, B}}
& \in L^r ( d (a_0, b_0) \otimes d \rho_N ) 
\\ \label{GABSCONTI2}
 \exp \left( \frac{1}{2} \int \phi  \, \psi^2dx \right) \chi_{\Omega_B} 
 & \in L^r (d (a_0, b_0) \otimes d \rho).
\end{align}
\end{lemma}

\noindent
In particular, $d\mu$ is a probability measure.
Moreover, we have  $d\mu_N \ll d(a_0, b_0) \otimes d\rho_N$ and   $d\mu \ll d(a_0, b_0) \otimes d\rho$.
For our application (in Theorem \ref{THM:GWP1}), we assume that $u_0$ (and $u(t)$ for any $t$) has mean 0.
Hence, in the following, we let $d a_0$ in \eqref{GmuN} and \eqref{Gmu}
to be the delta measure at $n = 0$ rather than the Lebesgue measure on $\mathbb{C}$.
Note that $da_0$ plays no significant role in any case. 

Finally, define $\Omega_{N, B} (s_1, s_2, K)$ and $\Omega_{B} (s_1, s_2, K)$ by

\begin{align*}
 \Omega_{N, B} (s_1, s_2, K) & = \Big\{  (a_n, b_n)_{ 0\leq n \leq N }  \in \Omega_{N, B}: 
\Big\| \sum_{ |n| \leq N} (a_n, b_n) e^{i n x} \Big\|_{H^{s_1, s_2} }  \leq K \Big\}  
\\ 
 \Omega_{ B} (s_1, s_2, K) & = \Big\{  (a_n, b_n)_{ n \geq 0 }  \in \Omega_{ B}: 
\Big\| \sum_{ n \in \mathbb{Z} } (a_n, b_n) e^{i n x} \Big\|_{H^{s_1, s_2} }  \leq K \Big\}.  
\end{align*}

\noindent
Then, we have 
\begin{lemma} [tightness] \label{LEM:tight}
Let $0 < s_1 < \frac{1}{2} < s_2 < 1$.
Then, for large $K > 0$, there exists $c > 0$,  independent of $N$,  such that
\begin{align}
\label{tight1}
 \mu_N \big(\Omega_{N, B} \setminus \Omega_{N, B} (s_1, s_2, K) \big)  \leq e^{-cK^2} 
\text{ and } \
 \mu \big(\Omega_{ B} \setminus  \Omega_{ B} (s_1, s_2, K) \big)  \leq e^{-cK^2}.
\end{align}
\end{lemma}

\noindent
The proof is analogous to that of Lemma \ref{LEM:GABSCONTI} in \cite{BO4}.
We prove Lemma \ref{LEM:tight} in Appendix.

\section{Function Spaces and Linear Estimates}
In this section, we go over the basic function spaces and linear estimates 
needed to establish Theorems \ref{THM:LWP1}.
First, we define a variant of the Bourgain spaces for $H^{s_1, s_2}$ defined in \eqref{GibbsIV}.
Let $X^{s_2, \infty,  b}$ and $X_\al^{s_2, \infty,  b}$ be the space given by the norms
\begin{align*}
& \|u\|_{X^{s_2, \infty,  b}}  = \| \jb{n}^{s_2} \jb{\tau - n^3}^b \ft{u}(n, \tau) \|_{L^\infty_n L^2_\tau}
\\ 
&  \|v\|_{X_\al^{s_2, \infty,  b}}=  \| \jb{n}^{s_2} \jb{\tau - \al n^3}^b \ft{v}(n, \tau) \|_{L^\infty_n L^2_\tau}.
\end{align*}

\noindent
Recall that  when $b > \frac{1}{2}$, the $X^{s_1, b} \times X_\al^{s_1, b}$ norm controls the $C([-T, T] ; H^{s_1} \times H^{s_1})$ norm.
This, however, does not hold when $b = \frac{1}{2}$. 
Now, define  $Y^{s_1, s_2}$  and $Y_\al^{s_1, s_2}$ where the norm is given by
\begin{align*}
\| u \|_{Y^{s_1, s_2}} = \|u\|_{Y^{s_1}} + \|u\|_{Y^{s_2, \infty}},
\text{ and } \| v \|_{Y_\al^{s_1, s_2}} = \|v\|_{Y_\al^{s_1}} + \|v\|_{Y_\al^{s_2, \infty}}, 
\end{align*}

\noindent
where
\begin{equation*}
\|u\|_{Y^{s_1}} = \|u \|_{X^{s_1, \frac{1}{2}}} + \| \jb{n}^{s_1} \ft{u}(n, \tau) \|_{L^2_n L^1_\tau},
\text{ and }
\|u\|_{Y^{s_2, \infty}} = \|u \|_{X^{s_2, \infty, \frac{1}{2}}}+ \|\jb{n}^{s_2} \ft{u}(n, \tau)\|_{L^\infty_n L^1_\tau}.
\end{equation*}

\noindent
$Y_\al^{s_1}$ and $Y_\al^{s_2, \infty}$ for $v$ are analogously defined with the obvious change of $\tau - n^3$ by $ \tau - \al n^3$.
Recall (c.f. \cite{CKSTT4})
that the $Y^{s_1} \times Y_\al^{s_1}$ norm controls the $C([-T, T] ; H^{s_1} \times H^{s_1})$ norm.
Also, we have
$\sup_n \jb{n}^{s_2} | \ft{u}(n, t)| 
\leq \sup_n \jb{n}^{s_2} \int | \ft{u}(n, \tau)| d \tau \leq \| u \|_{Y^{s_2, \infty}} $
for any $t \in \mathbb{R}$.
Hence, the $Y^{s_1, s_2} \times Y_\al^{s_1, s_2} $ norm controls 
the $C([-T, T] ; H^{s_1, s_2} \times H^{s_1, s_2})$ norm.
We prove Theorem \ref{THM:LWP1} by  a contraction argument on a ball in $Y^{s_1, s_2} \times Y_\al^{s_1, s_2}$
for appropriate $s_1$, $s_2$.

Next, we discuss the linear estimates.
By writing \eqref{MB} in the integral form, we see that $(u, v)$ is a solution to \eqref{MB}
with the initial condition $(u_0, v_0)$ for $|t| \leq T \leq 1$ if and only if
\begin{equation*} 
\begin{pmatrix} 
u(t) \\ v(t) 
\end{pmatrix} 
=
\begin{pmatrix} 
\vphantom{\Big|} \eta(t) S(t) u_0 - \eta_{_T}(t) \int_0^t S(t - t') \dx \big( \frac{v^2}{2} \big) (t') dt' \\
\vphantom{\Big|} \eta(t) S_\al(t) v_0 - \eta_{_T}(t) \int_0^t S_\al(t - t') \dx \big( u v \big) (t') dt' \\
\end{pmatrix} ,
\end{equation*}

\noindent
where $S(t) = e^{-t \dx^3}$ and $S_\al(t) = e^{-\al t \dx^3}$.
First, note that  $(\eta(t) S(t) u_0)^{\wedge}(n, \tau) = \ft{\eta}(\tau - n^3) \ft{u_0}(n)$
and $(\eta(t) S_\al(t) v_0)^{\wedge}(n, \tau) = \ft{\eta}(\tau - \al n^3) \ft{v_0}(n)$.

\begin{lemma} \label{LEM:hom} The following estimates hold for any $s_1, s_2, b \in \R$.
\[ \|  \eta(t)S(t) u_0  \|_{Y^{s_1, s_2}} \lesssim \| u_0  \|_{H^{s_1, s_2}}, \text{ and } \ 
\| \eta(t)S_\al(t) v_0  \|_{Y_\al^{s_1, s_2}}  \lesssim \| v_0  \|_{H^{s_1, s_2}}.  \]
\end{lemma}

\noindent
Now, let $-\frac{1}{2} < b' \leq 0 \leq b \leq b'+1$ and $T\leq 1$. 
Then, from (2.25) in Lemma 2.1 (ii) in Ginibre-Tsutsumi-Velo \cite{GTV}, we have
\begin{equation} \label{Duhamel1}
  \| \eta_{_T}  ( S *_R  F)  \|_{X^{s_1, b}} \lesssim T^{1-b + b'}\| F\|_{X^{s_1, b'}},
\end{equation}

\noindent
where $S*_R F(t) = \int_0^t S(t - t')  F (t')  dt'$.
From \cite[Lemma 7.2] {CKSTT4}, we have 
\begin{equation} \label{Duhamel2}  
\| \eta ( S*_R  F)  \|_{Y^{s_1}} \lesssim \| F\|_{Z^{s_1}}, \text{ and }
 \| \eta (S_\al*_R  G)\|_{Y_\al^{s_1}} \lesssim \|  G\|_{Z_\al^{s_1}}, 
\end{equation}
where
\[ \| u \|_{Z^{s_1}} = \| u \|_{X^{s_1, -\frac{1}{2}} 
}+ \big\| \jb{n}^{s_1} \jb{\tau - n^3}^{-1} \ft{u}(n, \tau) \big\|_{L^2_{n}L^1_\tau}\]
and $Z^{s_1}_\al$ is analogously defined with the change of $\tau - n^3$ by $\tau - \al n^3$.
Recall from \cite{BO1},  \cite{GTV} that a factor of $T^{0-}$ appears on the right hand sides of \eqref{Duhamel2}
if we replace $\eta$ by $\eta_{_T}$ in \eqref{Duhamel2}.
By the standard computation \cite{BO1}, we have
\begin{align*}
\eta ( t )  (S*_R F) (t)  
\sim \  & \eta(t) \sum_{k \geq 1} \frac{i^k t^k}{k!} \sum_{n \in \mathbb{Z}} e^{i(nx + n^3 t)} \int \eta(\tau - n^3) \big(\tau-n^3\big)^{k-1} \ft{F}(n, \tau) d \tau \\
& + \eta(t) \sum_{n \in \mathbb{Z}} e^{i nx } \int \frac{\big(1 - \eta\big) (\tau - n^3)}{\tau - n^3} e^{i \tau t} \ft{F}(n, \tau) d \tau \\
& + \eta(t) \sum_{n \in \mathbb{Z}} e^{i (nx+ n^3 t )} \int \frac{\big(1 - \eta\big) (\tau - n^3)}{\tau - n^3}  \ft{F}(n, \tau) d \tau \\
& =: \I + \II + \III.
\end{align*}

\noindent
Then, a direct computation shows that 
\begin{equation*}
\begin{cases}
\vphantom{\Big|}\|  \I  \|_{Y^{s_2, \infty}}, \, \|  \III  \|_{Y^{s_2, \infty}}  \lesssim \| \jb{n}^{s_2} \jb{\tau - n^3}^{-1} \ft{F}(n, \tau) \|_{L^\infty_n L^2_\tau}\\
\|  \II  \|_{Y^{s_2, \infty}}  \lesssim \|F\|_{X^{s_2, \infty, -\frac{1}{2}}} +  \| \jb{n}^{s_2} \jb{\tau - n^3}^{-1} \ft{F}(n, \tau) \|_{L^\infty_n L^2_\tau}.
\end{cases}
\end{equation*}

\noindent
i.e. we have 
\begin{equation} \label{Duhamel3}
\| \eta (S*_R F)\|_{Y^{s_2, \infty}}  \lesssim \|  F\|_{Z^{s_2, \infty}},
\end{equation}

\noindent
where 
\[ \| u \|_{Z^{s_2, \infty}} = \| u \|_{X^{s_2, \infty, -\frac{1}{2}}}
+ \| \jb{n}^{s_2} \jb{\tau - n^3}^{-1} \ft{u}(n, \tau) \|_{L^\infty_n L^2_\tau}.\]

\noindent
As before,  a factor of $T^{0-}$ appears on the right hand sides of \eqref{Duhamel3}
if we replace $\eta$ by $\eta_{_T}$ in \eqref{Duhamel3}.
Similarly, we have
$\| \eta (S_\al*_R G) \|_{Y_\al^{s_2, \infty}}  \lesssim \|  G \|_{Z_\al^{s_2, \infty}},$
where $Z_\al^{s_2, \infty}$ is analogously defined with the change of $\tau - n^3$ by $ \tau - \al n^3$.
\noindent
Lastly, define $Z^{s_1, s_2}$ and $Z_\al^{s_1, s_2}$ by 
\[ \| u \|_{Z^{s_1, s_2}} = \|u \|_{Z^{s_1}} +\|u \|_{Z^{s_2, \infty}} \text{ and }
\| v \|_{Z_\al^{s_1, s_2}} = \|v \|_{Z_\al^{s_1}} +\|v \|_{Z_\al^{s_2, \infty}}. \] 

\noindent
Then, we have the following Duhamel estimates for $b = \frac{1}{2}$ and $b' = -\frac{1}{2}$.

\begin{lemma} \label{LEM:duhamel2}
Let $0< T \leq 1$. Then, we have
\begin{equation*} 
\| \eta_{_T} (S*_R F)\|_{Y^{s_1, s_2}}  \lesssim T^{0-} \|  F\|_{Z^{s_1, s_2}}
\text{ and }\| \eta_{_T} (S_\al*_RG)\|_{Y_\al^{s_1, s_2}}  \lesssim T^{0-} \|  G\|_{Z_\al^{s_1, s_2}}.
\end{equation*}
\end{lemma}

\section{New Local Well-Posedness Result for $\al \in (0, 4) \setminus\{1\}$}

In this section, we prove Theorem \ref{THM:LWP1} 
by constructing a contraction in $Y^{s_1, s_2} \times Y_\al^{s_1, s_2}$,
where  $\frac{1}{4} < s_1 < \frac{1}{2} < s_2 < 1 $ with $2s_1 > s_2 > \frac{1}{2} + \frac{1}{2}\max( \nu_{c_1},\nu_{d_1}, \nu_{d_2}) $ and 
$\max( \nu_{c_1},\nu_{d_1}, \nu_{d_2}) < 1$.
Once we prove the bilinear estimates
\begin{align} 
\| \dx( v_1 v_2 ) \|_{Z^{s_1, s_2}} & \lesssim \|v_1 \|_{Y_\al^{s_1, s_2}} \|v_2 \|_{Y_\al^{s_1, s_2}} 
\label{Xbilinear1} \\
\| \dx( u v) \|_{Z_\al^{s_1, s_2}} & \lesssim \| u \|_{Y^{s_1, s_2}} \| v \|_{Y_\al^{s_1, s_2}} \label{Xbilinear2}
\end{align}

\noindent
with the mean 0 condition on $u$, 
the local well-posedness of \eqref{MB} in $H^{s_1, s_2} \times H^{s_1, s_2}$ on a time interval of size $\sim 1$ follows
from Lemmata \ref{LEM:hom}, \ref{LEM:duhamel2},  \eqref{Xbilinear1}, and \eqref{Xbilinear2},
provided that $\|(u_0, v_0)\|_{H^{s_1, s_2} \times H^{s_1, s_2}} $ is sufficiently small.

To establish the LWP for the general initial data $(u_0, v_0)$ without the smallness assumption, 
we need to gain a positive power of $T$ from the bilinear estimates \eqref{Xbilinear1} and \eqref{Xbilinear2},
assuming that the functions are supported on the time interval $[-2T, 2T]$. 
In particular, we need to prove
\begin{align} 
 \left\| \eta_{_{2T}} \dx( v_1 v_2 ) \right\|_{Z^{s_1, s_2}} & \lesssim T^\theta\|v_1 \|_{Y_\al^{s_1, s_2}} \|v_2 \|_{Y_\al^{s_1, s_2}} 
\label{Xbilinear3} \\
 \left\| \eta_{_{2T}} \dx( u v) \right\|_{Z_\al^{s_1, s_2}} & \lesssim T^\theta \| u \|_{Y^{s_1, s_2}} \| v \|_{Y_\al^{s_1, s_2}},
\label{Xbilinear4}
\end{align}

\noindent
for some $\theta > 0$.
For the rest of this section, we first present the proof of \eqref{Xbilinear1} and \eqref{Xbilinear2}
in Propositions \ref{Gibbsbilinear1}, \ref{Gibbsbilinear3}, and \ref{Gibbsbilinear2}.
Then, we mention how to obtain the positive power of $T$ as in \eqref{Xbilinear3} and \eqref{Xbilinear4}.

First, recall the following result in \cite{OH1}. (See Remark \ref{REM:bilinear}.)

\begin{lemma} \label{LEM:oldbilinear} Let $s_1 \geq 0$.
Then, we have 
\begin{equation} \label{Xbilinear5}
\left\| \dx( v_1 v_2 ) \right\|_{Z^{s_1}} \lesssim \|v_1 \|_{Y_\al^{s_1}} \|v_2 \|_{Y_\al^{s_1}} 
\end{equation}

\noindent
on $\{|n| \lesssim 1\}$ or 
\begin{equation} \label{nonres1}
A = \{ (n, n_1, n_2) : n = n_1 + n_2, |n| \gtrsim 1, |n_1 - c_1 n| \geq 1 \text{ and } |n_1 - c_2 n| \geq 1\}, 
\end{equation}
where $n,$ $ n_1$, and $n_2$ are the spatial Fourier variables of $ v_1 v_2 $, $v_1$, and $v_2$.
Also, we have 
\begin{equation} \label{Xbilinear6}
\left\| \dx( u v) \right\|_{Z_\al^{s_1}} \lesssim \| u \|_{Y^{s_1}} \| v \|_{Y_\al^{s_1}},
\end{equation}

\noindent
on  $\{|n| \lesssim 1\}$ or 
\begin{equation} \label{nonres2}
C = \{ (n, n_1, n_2) : n = n_1 + n_2, |n| \gtrsim 1, |n_1 - d_1 n| \geq 1 \text{ and } |n_1 - d_2 n| \geq 1\},
\end{equation}
where $n,$ $ n_1$, and $n_2$ are the spatial Fourier variables of $ uv $, $u$, and $v$.

\end{lemma}

\noindent
We point out that the proof of \eqref{Xbilinear5} and \eqref{Xbilinear6}
are basically the same as that of the bilinear estimate for KdV for $s \geq 0$ in \cite{BO1}.
Hence, by assuming that $v_1v_2$ in \eqref{Xbilinear5} and $uv$ in \eqref{Xbilinear6}
are supported on time interval $[-2T, 2T]$, we gain a positive power $T^\theta$
on the right hand sides.  
For details, see \cite{BO1}.

Now, we prove \eqref{Xbilinear1} in Propositions \ref{Gibbsbilinear1} and \ref{Gibbsbilinear3}.

\begin{proposition} \label{Gibbsbilinear1}
Assume $\nu_{c_1} <1$.
Then, for  $\frac{1}{4} \leq s_1 < \frac{1}{2} < s_2 < 1$ with $ s_2 > \frac{1}{2} + \frac{1}{2}\nu_{c_1}$, we have
\begin{equation}
\left\| \dx( v_1 v_2 ) \right\|_{Z^{s_1}} \lesssim \|v_1 \|_{Y_\al^{s_1, s_2}} \|v_2 \|_{Y_\al^{s_1, s_2}} .
\end{equation}
\end{proposition}

\begin{proof}  
In view of Lemma \ref{LEM:oldbilinear}, we restrict our attention 
to  
\begin{equation} \label{res1}
 B = \{ (n, n_1, n_2) \in \mathbb{Z}^3 : n = n_1 + n_2, |n|\gtrsim 1, | n_1 - c_1n| < 1 \text{ and } |n_1 - c_2 n| < 1 \}.
\end{equation}

\noindent
For fixed $n\in \mathbb{Z}$, there are only 4 values of $n_1$ in $B$, i.e.  $n_1 = [c_1 n ] , [c_1 n ] + 1, [c_2 n ] ,$ or $ [c_2 n ] + 1$,
where $[ \, \cdot \, ]  $ is the integer part function. 
Thus, there are 4 terms contributing in the convolution in the spatial Fourier variable.
Note that we have $|n| \sim |n_1| \sim |n_2|$ on $B$.
By the definition of the minimal type index $\nu_{c_1}$ (see \eqref{lowerbdC}), we have 
\begin{equation} \label{Gresonance1}
\MAX: = \max ( \jb{\tau - n^3}, \jb{\tau_1 - \al n_1^3}, \jb{\tau_2 - \al n_2^3} )\gtrsim |n^3 - \al n_1^3 - \al n_2^3| \gtrsim |n|^{1-\nu_{c_1}-\eps} 
\end{equation}

\noindent
for any $\eps > 0$.
Without loss of generality, assume $\jb{\tau - n^3}, \jb{\tau_1 - \al n_1^3} \gtrsim \jb{\tau_2 - \al n_2^3}$.

First, consider the $X^{s_1, - \frac{1}{2}}$ part of the $Z^{s_1}$ norm. 
It suffices to show
\begin{equation} \label{timelocal1}
\bigg\| \intt_{\tau = \tau_1 + \tau_2}  \frac{1}{\jb{\tau-n^3}^\frac{1}{2}}
\frac{\jb{n}^{s_1+1}}{\jb{n_1}^{s_1} \jb{n_2}^{s_2}} \frac{f(n_1, \tau_1)}{\jb{\tau_1 - \al n_1^3}^\frac{1}{2}}  
g(n_2, \tau_2) d \tau_1 \bigg\|_{L^2_{n, \tau}} \lesssim \|f\|_{L^2_{n, \tau}}\|g\|_{L^\infty_{n}L^1_{\tau}} 
\end{equation}

\noindent
 for each $n_1 = [c_1 n ] , [c_1 n ] + 1, [c_2 n ] ,$ or $ [c_2 n ] + 1$.
From  \eqref{Gresonance1}, we have 
\begin{equation} \label{timelocal2}
\frac{\jb{n}^{s_1+1}}{\jb{n_1}^{s_1} \jb{n_2}^{s_2}} \frac{1}{\MAX^\frac{1}{2}}  
\lesssim |n|^{-s_2 + \frac{1}{2}+\frac{1}{2}\nu_{c_1} + \frac{1}{2}\eps} \lesssim 1 
\end{equation}
for $ s_2 \geq \tfrac{1}{2}+\tfrac{1}{2}\nu_{c_1} + \tfrac{1}{2}\eps$. 
The rest follows from H\"older inequality in $n$ and Young's inequality in $\tau$. 

\smallskip

Now, consider the second part of the $Z^{s_1}$ norm. 
It suffices to show
\[\bigg\| \intt_{\tau = \tau_1 + \tau_2}  \frac{1}{\jb{\tau-n^3}}
\frac{\jb{n}^{s_1+1}}{\jb{n_1}^{s_1} \jb{n_2}^{s_2}} \frac{f(n_1, \tau_1)}{\jb{\tau_1 - \al n_1^3}^\frac{1}{2}}  
g(n_2, \tau_2) d \tau_1 \bigg\|_{L^2_n L^1_{\tau}} \lesssim \|f\|_{L^2_{n_, \tau}}\|g\|_{L^\infty_{n}L^1_{\tau}}  \] 

\noindent
 for $n_1 = [c_1 n ] , [c_1 n ] + 1, [c_2 n ] ,$ or $ [c_2 n ] + 1$.
By H\"older's inequality in $\tau$, we have
\[\text{LHS}  \leq c\,  \bigg\| \int \frac{1}{\jb{\tau-n^3}^{\frac{1}{2}-}}
\frac{\jb{n}^{s_1+1}}{\jb{n_1}^{s_1} \jb{n_2}^{s_2}} \frac{f(n_1, \tau_1)}{\jb{\tau_1 - \al n_1^3}^\frac{1}{2}}  
g(n_2, \tau_2) d \tau_1 \bigg\|_{L^2_{n, \tau}}, \]
where $c = \sup_{n} \big(\int \jb{\tau - n^3}^{-1-} d\tau \big)^{1/2} < \infty $.
The rest follows from the previous part
as long as $s_2 > \tfrac{1}{2}+\tfrac{1}{2}\nu_{c_1} + \tfrac{1}{2}\eps$.
\end{proof}

\begin{proposition} \label{Gibbsbilinear3}
Assume $\nu_{c_1} <1$.
Then, for  $\frac{1}{4} \leq s_1 < \frac{1}{2} < s_2 < 1$ with $2s_1 > s_2 > \frac{1}{2} + \frac{1}{2}\nu_{c_1}$, we have
\begin{equation}
\left\| \dx( v_1 v_2 ) \right\|_{Z^{s_2, \infty}} \lesssim \|v_1 \|_{Y_\al^{s_1, s_2}} \|v_2 \|_{Y_\al^{s_1, s_2}} .
\end{equation}
\end{proposition}

\begin{proof}
For $|n|\lesssim 1 $, we have $\jb{n}^{s_2} \sim \jb{n}^{s_1}$ and $L^\infty_n$-norm $\sim$ $L^2_n$-norm. 
i.e. it reduces to Proposition \ref{Gibbsbilinear1}.
Thus, assume $|n| \gtrsim 1$.
Without loss of generality, assume $\jb{\tau - n^3}, \jb{\tau_1 - \al n_1^3} \gtrsim \jb{\tau_2 - \al n_2^3}$.
First, consider the $X^{s_2, \infty, -\frac{1}{2}}$ part of the $Z^{s_2, \infty}$ norm.

\noindent
$\bullet$ {\bf Case (1):} Away from resonances, i.e. on $A$ in \eqref{nonres1}
\nopagebreak

From \cite{OH1}, we have 
\begin{equation} \label{Gnonresonance1}
\MAX: = \max ( \jb{\tau - n^3}, \jb{\tau_1 - \al n_1^3}, \jb{\tau_2 - \al n_2^3} )\gtrsim |n^3 - \al n_1^3 - \al n_2^3| \gtrsim n^{2}  .
\end{equation}

\noindent
This can be seen from the fact that $P_{n}(n_1) := n^3 - \al n_1^3 - \al n_2^3$ is a quadratic polynomial in $n_1$ for fixed $n$
and that  $\partial_{n_1} P_{n}(n_1)$ at $n_1 = c_1 n, c_2n$ (i.e. at the zeros of $P_{n}(n_1)$) is of order $n^2$.
It suffices to  show 
\[\bigg\| \sum_{n = n_1 + n_2}\intt_{\tau = \tau_1 + \tau_2} 
\frac{\jb{n}^{s_2+1}}{\jb{n_1}^{s_1} \jb{n_2}^{s_1}} \frac{f(n_1, \tau_1)g(n_2, \tau_2) d \tau_1}
{\jb{\tau-n^3}^\frac{1}{2}\jb{\tau_1-\al n_1^3}^\frac{1}{2}}
 \bigg\|_{L^\infty_n L^2_\tau} \lesssim \|f\|_{L^2_{n_1, \tau_1}}\|g\|_{L^2_{n_2}L^1_{\tau_2}} . \] 

If $ |n_1|, |n_2| \gtrsim |n|$, then we have, from \eqref{Gnonresonance1}, 
$\frac{\jb{n}^{s_2+1}}{\jb{n_1}^{s_1} \jb{n_2}^{s_1}} \frac{1}{\MAX^\frac{1}{2}}  \lesssim |n|^{s_2 - 2 s_1} \lesssim 1$
 for  $2 s_1 \geq s_2. $
Otherwise, we have $|n_1| \ll |n|$ or $|n_2| \ll |n|$.
In this case, we have
$\MAX \gtrsim |n^3 - \al n_1^3 - \al n_2^3| \gtrsim |n|^{3}  $
and this gives us
$\frac{\jb{n}^{s_2+1}}{\jb{n_1}^{s_1} \jb{n_2}^{s_1}} \frac{1}{\MAX^\frac{1}{2}}  
\lesssim |n|^{s_2 - s_1 -\frac{1}{2}} \lesssim 1$
for  $ s_1+ \frac{1}{2} \geq s_2. $
Then, the rest follows from Young's inequality in $n$ and $\tau$.

\noindent
$\bullet$ {\bf Case (2):} Near resonances, i.e. on $B$ in \eqref{res1}
\nopagebreak

It suffices  to show, for $n_1 = [c_1 n ] , [c_1 n ] + 1, [c_2 n ] ,$ or $ [c_2 n ] + 1$, 
\[\bigg\| \intt_{\tau = \tau_1 + \tau_2} 
\frac{\jb{n}^{s_2+1}}{\jb{n_1}^{s_2} \jb{n_2}^{s_2}} \frac{f(n_1, \tau_1)g(n_2, \tau_2) d \tau_1}{\jb{\tau-n^3}^\frac{1}{2}\jb{\tau_1-\al n_1^3}^\frac{1}{2}}
 \bigg\|_{L^\infty_n L^2_\tau} \lesssim \|f\|_{L^\infty_{n_1} L^2_{\tau_1}}\|g\|_{L^\infty_{n_2}L^1_{\tau_2}} . \] 
\noindent
From  \eqref{Gresonance1}, we have
$\frac{\jb{n}^{s_2+1}}{\jb{n_1}^{s_2} \jb{n_2}^{s_2}} \frac{1}{\MAX^\frac{1}{2}}  
\lesssim |n|^{-s_2 + \frac{1}{2}+\frac{1}{2}\nu_{c_1} + \frac{1}{2}\eps} \lesssim 1 $
for $ s_2 \geq \tfrac{1}{2}+\tfrac{1}{2}\nu_{c_1} + \tfrac{1}{2}\eps$. 
Then, the rest follows 
from H\"older inequality in $n$ and Young's inequality in $\tau$.

Now, consider the $L^\infty_n L^1_\tau$ part of $Z^{s_2, \infty}$ norm.

\noindent
$\bullet$ {\bf Case (3):} Away from resonance.
\nopagebreak

It suffices to show
\[\bigg\| \sum_{n = n_1 + n_2}\intt_{\tau = \tau_1 + \tau_2} 
\frac{\jb{n}^{s_2+1}}{\jb{n_1}^{s_1} \jb{n_2}^{s_1}} 
\frac{f(n_1, \tau_1)g(n_2, \tau_2) d \tau_1}{\jb{\tau-n^3}\jb{\tau_1-\al n_1^3}^\frac{1}{2}}
\bigg\|_{L^\infty_n L^1_\tau} \lesssim \|f\|_{L^2_{n_1, \tau_1}}\|g\|_{L^2_{n_2}L^1_{\tau_2}} . \] 

\noindent
As in the proof of Proposition \ref{Gibbsbilinear1}, 
apply H\"older's inequality in $\tau$, 
and the rest follows from Case (1) for $2 s_1 > s_2$.

\noindent
$\bullet$ {\bf Case (4):} Near resonances.
\nopagebreak

In this case,  it suffices  to show, for $n_1 = [c_1 n ] , [c_1 n ] + 1, [c_2 n ] ,$ or $ [c_2 n ] + 1$, 
\[\bigg\| \intt_{\tau = \tau_1 + \tau_2} 
\frac{\jb{n}^{s_2+1}}{\jb{n_1}^{s_2} \jb{n_2}^{s_2}} 
\frac{f(n_1, \tau_1)g(n_2, \tau_2) d \tau_1}{\jb{\tau-n^3}\jb{\tau_1-\al n_1^3}^\frac{1}{2}}
\bigg\|_{L^\infty_n L^1_\tau} \lesssim \|f\|_{L^\infty_{n_1} L^2_{\tau_1}}\|g\|_{L^\infty_{n_2}L^1_{\tau_2}} . \] 
As in Case(3), apply H\"older's inequality in $\tau$, and the rest follows from Case (2)  
as long as for $ s_2 > \tfrac{1}{2}+\tfrac{1}{2}\nu_{c_1} + \tfrac{1}{2}\eps$. 
\end{proof}

\begin{proposition}  \label{Gibbsbilinear2}
Assume $ \max( \nu_{d_1}, \nu_{d_2} )<1$ and the mean 0 condition for $u$.
Then, for  $\frac{1}{4} \leq  s_1 < \frac{1}{2} < s_2 < 1$ 
with $s_2 > \frac{1}{2} + \frac{1}{2}\max( \nu_{d_1}, \nu_{d_2} )$, we have
\begin{equation}
\left\| \dx( u v) \right\|_{Z_\al^{s_1, s_2}} \lesssim \| u \|_{Y^{s_1, s_2}} \| v \|_{Y_\al^{s_1, s_2}}.
\end{equation}
\end{proposition}

\begin{proof}
We omit the details
of the proof of Proposition \ref{Gibbsbilinear2}
since it is basically the same as those of Propositions \ref{Gibbsbilinear1} and \ref{Gibbsbilinear3}
once we point out the following. 
Let $\MAX := \max ( \jb{\tau - \al n^3}, \jb{\tau_1 - n_1^3}, \jb{\tau_2 - \al n_2^3} )$.
Then, we have 
$\MAX \gtrsim | \al n^3 -  n_1^3 - \al n_2^3| \gtrsim | n_1 n|$
on $C$ in \eqref{nonres2}, i.e. away from resonances. 
Moreover, if $|n_1| \ll |n|$ or $|n_2| \ll |n|$, then we have 
$\MAX \gtrsim | \al n^3 -  n_1^3 - \al n_2^3| \gtrsim |n_1n^2| .$
Now, define the resonance set $D$ by
\begin{equation} \label{res2}
D  = \{ (n, n_1, n_2) : |n|\gtrsim 1,  | n_1 - d_1 n | < 1 \text{ and } |n_1 - d_2 n| < 1 \}.
\end{equation} 

\noindent
i.e. the left hand side of \eqref{resonance2} can be small on $D$. 
As before,  for fixed $n\in \mathbb{Z}$, there are only 4 values of $n_1$ in $D$, 
i.e. $n_1 = [d_1 n ] , [d_1 n ] + 1, [d_2 n ] ,$ or $ [d_2 n ] + 1$.
Thus, there are 4 terms contributing in the convolution in the spatial Fourier variable.
Note that we have $|n| \sim |n_1| \sim |n_2|$ on $D$.
By the definition of the minimal type indices $\nu_{d_1}, \nu_{d_2}$ (see \eqref{lowerbdD}), we have 
\begin{align*} 
\MAX: = 
 \gtrsim & | \al n^3 -  n_1^3 - \al n_2^3| 
 \gtrsim|n_1|  |n|^{0-} \gtrsim |n|^{1-\max(\nu_{d_1}, \nu_{d_2}) - \eps} .
\end{align*}

\noindent
 for any $\eps > 0$.
 The rest follows as in the proof of 
 Propositions \ref{Gibbsbilinear1} and \ref{Gibbsbilinear3}
\end{proof}

This establishes the LWP for the periodic Majda-Biello system \eqref{MB}
for  small initial data $(u_0, v_0) \in H^{s_1, s_2} \times H^{s_1, s_2}$. 
For the general data without the smallness assumption, 
one can exploit small time intervals $[-2T, 2T]$, $T \ll 1$ to gain an extra factor
$T^\theta$ for some $\theta > 0$ as in \eqref{Xbilinear3} and \eqref{Xbilinear4}.  
We discuss how to gain $T^\theta$ in \eqref{timelocal1} assuming the functions are localized in time,
i.e. by replacing $f$ (or  $g$) by $\ft{\eta_{_{2T}}}*f$ (or  $\ft{\eta_{_{2T}}}*g$) in \eqref{timelocal1}.

Suppose $\MAX = \jb{\tau - n^3}$.
Then, it suffices to prove
\begin{equation} \label{timelocal3}
\text{LHS of } \eqref{timelocal1} \text{ with $f$ replaced by } \ft{\eta_{_{2T}}}*f   \lesssim T^\theta \|f\|_{L^2_{n} L^1_\tau} \|g\|_{L^\infty_{n}L^1_{\tau}}.
\end{equation}

\noindent
By \eqref{timelocal1}, H\"older in $n$, and Young's inequality in $\tau$, we have 
$\text{LHS of } \eqref{timelocal3} \lesssim \|\ft{\eta_{_{2T}}}*f\|_{L^2_{n, \tau}} \|g\|_{L^\infty_{n}L^1_{\tau}}$,
and the first factor is bounded by 
$ \| \ft{\eta_{_{2T}}}\|_{L^2_\tau}  \|f\|_{L^2_{n} L^1_\tau}
\sim T^\frac{1}{2} \|f\|_{L^2_{n} L^1_\tau}$
by Young's inequality. 
Next, suppose $\MAX = \jb{\tau_1 - \al n_1^3}$.
Then, it suffices to prove
\begin{equation} \label{timelocal4}
\text{LHS of } \eqref{timelocal1} 
\text{ with $g$ replaced by } \ft{\eta_{_{2T}}}*g   
\lesssim T^\theta \|f\|_{L^2_{n, \tau} } \|g\|_{L^\infty_{n}L^1_{\tau}}.
\end{equation}

\noindent
By \eqref{timelocal1}, H\"older in $n, \tau$, and Young's inequality in $\tau$, we have 
\begin{align*}
\text{LHS}& \text{ of }  \eqref{timelocal4} \lesssim 
\| \jb{\tau - n^3}^{-\frac{1}{2}} f * (\ft{\eta_{_{2T}}}*g) \|_{L^2_{n, \tau}}
\leq \| \jb{\tau - n^3}^{-\frac{1}{2}}\|_{L^\infty_n L^3_\tau} \|f * (\ft{\eta_{_{2T}}}*g) \|_{L^2_n L^6_\tau} \\
& \lesssim \|f \|_{L^2_{n, \tau}} \|\ft{\eta_{_{2T}}}*g \|_{L^\infty_n L^{3/2}_\tau}
\leq \|\ft{\eta_{_{2T}}}\|_{L^{3/2}_\tau} \|f \|_{L^2_{n, \tau}}  \|g \|_{L^\infty_n L^1_\tau} 
\sim T^\frac{1}{3} \|f \|_{L^2_{n, \tau}}  \|g \|_{L^\infty_n L^1_\tau} .
\end{align*}

\noindent
All the other estimates in Propositions \ref{Gibbsbilinear1}, \ref{Gibbsbilinear3}, and \ref{Gibbsbilinear2}
can be modified in a similar manner to gain $T^\theta$
and we omit the detail.
(Note that
we have $L^2_\tau$ on the left hand side, 
possibly after H\"older inequality in $\tau$,
and $L^2_\tau$ and $L^1_\tau$ on the right hand side
for all the estimates.)
This yields the local well-posedness on the general data
$(u_0, v_0) \in H^{s_1, s_2} \times H^{s_1, s_2}$ without smallness assumption.
Note that the time interval $[-T, T]$ of existence depends 
on the size of $\| (u_0, v_0) \|_{H^{s_1, s_2} \times H^{s_1, s_2}}$ in an inverted polynomial way.
See \cite{BO1}.

\section{Global Well-Posedness on the Statistical Ensemble  and the \\ Invariance of the Gibbs Measure} 
\label{SEC:DIFFGIBBS}

Once we establish the local well-posedness of \eqref{MB} via the fixed point theorem
and tightness  of $\mu_N$ and $\mu$ (Lemma \ref{LEM:tight}),
Bourgain's argument in  \cite{BO4} yields the global well-posedness a.s. on the statistical ensemble.
As for the invariance of the Gibbs measure $\mu$ (Theorem \ref{THM:GIBBSinvariance}), 
we follow the argument in Rey-Bellet and Thomas \cite{RT}.
We include these arguments for the sake of completeness.

Consider the finite dimensional approximation to \eqref{MB}:
\begin{equation} \label{approxMB}
\left\{
\begin{array}{l}
u^N_t + u^N_{xxx} + \mathbb{P}_N(v^Nv^N_x) = 0\\
v^N_t + \al v^N_{xxx} + \mathbb{P}_N\big((u^Nv^N)_x\big) = 0 ,
\end{array}
\right.
\end{equation}
with $\big(u^N (x, 0), v^N(x, 0)\big) = \big(u_0^N(x), v_0^N(x) \big) = \big(\mathbb{P}_N u_0(x),  \mathbb{P}_N v_0(x)\big)$
for $N \to \infty$.
In the following, we assume that the mean of $u_0$ is 0.
Note that $\int (u^N)^2 + (v^N)^2 dx$ is conserved under the finite dimensional flow.
Moreover, the finite dimensional truncation of the Hamiltonian
\begin{equation} \label{GFINITEHAMIL1}
H_N(u, v) = \frac{1}{2} \int (u^N_x)^2 + \al (v^N_x)^2 - \mathbb{P}_N\big(u^N (v^N)^2 \big)dx
\end{equation}
is conserved as well.  
Therefore, by Liouville's Theorem, $\mu_N$ is invariant under the flow of \eqref{approxMB}.

In the following, we first establish the a.s. GWP of \eqref{approxMB} 
(with an explicit growth bound modulo small set), independent of $N$.
Then, using this and the invariance of $\mu_N$, we show the a.s. GWP of \eqref{MB} and 
the invariance of $\mu$.

\begin{lemma} \label{finitedimgrowth}
Let $\frac{1}{4} < s_1 < \frac{1}{2} < s_2 < 1$ with $ 2s_1 > s_2 > \frac{1}{2} + \frac{1}{2} \max( \nu_{c_1},\nu_{d_1},\nu_{d_2})$, 
$T < \infty$, and $\eps > 0$.
There exists a set $\Omega_{N, \eps} \subset  H^{s_1, s_2} \times H^{s_1, s_2}$ such that
$\mu_N ( \Omega_{N, \eps}^c) < \eps$ and for $(u^N_0, v^N_0) \in \Omega_{N, \eps}$,
the solution $(u^N, v^N)$ to the IVP \eqref{approxMB}
satisfies, for $|t| \leq T$, 
\[ \big\| \big(u^N, v^N \big) (t) \big\|_{H^{s_1, s_2} \times H^{s_1, s_2}} \lesssim \bigg( \log \frac{T}{\eps} \bigg)^\frac{1}{2}. \]

\end{lemma}

\begin{proof} 
Let $S_N(t)$ be the flow map corresponding to \eqref{approxMB}.
By Liouville's theorem, $\mu_N$ is invariant under $S_N(t)$ for all $t$ (as long as the solution exists.)
Note that $\mathbb{P}_N$ acts continuously on the function spaces used for the local theory of \eqref{MB} in $H^{s_1, s_2} \times H^{s_1, s_2}$.
Then, from the local well-posedness of \eqref{MB} in $H^{s_1, s_2} \times H^{s_1, s_2}$, 
we obtain the local well-posedness of \eqref{approxMB} in $H^{s_1, s_2} \times H^{s_1, s_2}$
with the same bound; i.e. if $(u^N_0, v^N_0) \in \Omega_{N, B} (s_1, s_2, K)$, then
\[ \big\| \big(u^N, v^N \big) (t) \big\|_{H^{s_1, s_2} \times H^{s_1, s_2}} \leq 2K\]
for $|t| \leq \dl \sim K^{-\theta}$ with some $\theta > 0$.  
Note that this is independent of $N$.

Let $S = S_N(\dl)$ and consider the set
$\Omega_{N, \eps} = \bigcap_{j = - [ \frac{T}{\dl}]}^{  [\frac{T}{\dl}]} S^j \Omega_{N, B} (s_1, s_2, K)$.
From the invariance of $\mu_N$, we have
\[\mu_N(\Omega_{N, \eps}^c) \leq \frac{T}{\dl} \, \mu_N\big((\Omega_{N, B} (s_1, s_2, K))^c\big) \sim T K^\theta e^{-cK^2}.\]
Hence,  we have $\mu_N(\Omega_{N, \eps}^c) < \eps$ for $K \sim \big(\log \frac{T}{\eps} \big)^\frac{1}{2}$.
If $(u^N_0, v^N_0) \in \Omega_{N, \eps}$, then 
by construction we have $\big\| \big(u^N, v^N \big) (j\dl) \big\|_{H^{s_1, s_2}} \leq K$
for $j = 0, 1, \cdots, [\frac{T}{\dl}].$
Thus, we have the well-posedness on each subinterval
$[j \dl, (j+1) \dl]$ of $[0, T]$ for $j = 0, 1, \cdots, [\frac{T}{\dl}] - 1$  (with bounds independent of $N$)
and 
\[ \big\| \big(u^N, v^N \big) (t) \big\|_{H^{s_1, s_2} } \leq 2K \sim \bigg(\log \frac{T}{\eps}\bigg)^\frac{1}{2} \text{ for } 0 \leq t \leq T.\]

\noindent
Since the flow is time-reversible, we have 
$ \big\| \big(u^N, v^N \big) (t) \big\|_{H^{s_1, s_2} } \lesssim \left(\log \frac{T}{\eps}\right)^\frac{1}{2}$
for  $ |t| \leq T$.
\end{proof}

\begin{corollary} \label{finitedimGWP}
Let $\frac{1}{4} < s_1 < \frac{1}{2} < s_2 < 1$ with $ 2s_1 > s_2 > \frac{1}{2} + \frac{1}{2} \max( \nu_{c_1},\nu_{d_1},\nu_{d_2})$
 and $\eps > 0$.
There exists a set $\Omega'_{N, \eps} \subset H^{s_1, s_2} \times H^{s_1, s_2}$ such that
$\mu_N \big( (\Omega_{N, \eps}')^c\big) < \eps$ and for $(u^N_0, v^N_0) \in \Omega'_{N, \eps}$, 
the solution $(u^N, v^N)$ to the IVP \eqref{approxMB}
satisfies, for all $t \in \mathbb{R}$, 
\begin{equation} \label{Ggrowth1}
 \big\| \big(u^N, v^N \big) (t) \big\|_{H^{s_1, s_2} \times H^{s_1, s_2}} \lesssim \bigg( \log \frac{1 + |t|}{\eps} \bigg)^\frac{1}{2}. 
\end{equation}
\end{corollary}

\begin{proof} 
With $T_j = 2^j$ and $\eps_j = \frac{\eps}{2^{j+1}}$, construct $\Omega^{(j)}_{N, \eps_j}$ described in Lemma \ref{finitedimgrowth}.
Then, let $\Omega'_{N, \eps} = \bigcap_{j = 1}^\infty \Omega^{(j)}_{N, \eps_j}$.
By construction, we have \eqref{Ggrowth1} for $(u^N_0, v^N_0) \in \Omega'_{N, \eps_j}$,
and 
$\mu_N \big( (\Omega_{N, \eps}')^c\big) \leq \sum_{j = 1}^\infty \mu_N\big( ( \Omega^{(j)}_{N, \eps_j} )^c \big) < \eps. $
\end{proof}

\begin{proposition}  \label{GWPensemble2}
Let $\frac{1}{4} < s_1 < \frac{1}{2} < s_2 < 1$ with $ 2s_1 > s_2 > \frac{1}{2} + \frac{1}{2} \max( \nu_{c_1},\nu_{d_1},\nu_{d_2})$
 and $\eps > 0$.
Then, there exists a set $\Omega_{ \eps} \subset  H^{s_1, s_2} \times H^{s_1, s_2}$ such that
$\mu ( \Omega_{ \eps}^c) < \eps$ and, 
for a set of data $(u_0, v_0) \in \Omega_\eps$,
the Majda-Biello system \eqref{MB} 
is globally well-posed with the bound
\[ \| (u, v ) (t) \|_{H^{s_1, s_2} \times H^{s_1, s_2}} \lesssim \bigg( \log \frac{1 + |t|}{\eps} \bigg)^\frac{1}{2} \text{ for all } t \in \mathbb{R}. \]

\end{proposition}

\begin{proof} 
First, fix $\s_1, \s_2$ such that $ \frac{1}{4} < s_1 < \s_1 < \frac{1}{2} < s_2 < \s_2 < 1$ with $2\s_1 > \s_2$.
Also, fix $T< \infty$, $\eps >0$, and large $N = N(T, \eps)$  (to be determined later.)
Consider \eqref{approxMB}
with $(U, V) = (u^N, v^N)$ and $(U ,V)|_{t = 0} = (U_0, V_0) = (\mathbb{P}_N u_0,  \mathbb{P}_N v_0)$.
Then, with $\g = \min ( \s_1 - s_1, \s_2 - s_2)>0$, we have
\begin{equation} \label{Ghomo1}
 \| (u_0, v_0) - (U_0, V_0) \|_{H^{s_1, s_2} } \lesssim N^{-\g} \| (u_0, v_0)  \|_{H^{\s_1, \s_2} }.
\end{equation}

As in the proof of Lemma \ref{finitedimgrowth}, construct the $\Omega_{N, \eps}$ set with  
the large radius $K \sim \big( \log \frac{T}{\eps^2} \big)^\frac{1}{2} $
such that 
$ d (a_0, b_0) \otimes \rho_N( \Omega_{N, \eps}^c) <  Z^{-\frac{1}{2}}{\eps^2} $.
Now, let $\wt{\Omega}_{\eps} = \{ (a_n, b_n)_{n\geq 0} \in \Omega_B (\s_1, \s_2, K): (a_n, b_n)_{0 \leq n \leq N} \in \Omega_{N, \eps} \}$.
Note that $\wt{\Omega}_{\eps}$ really depends on both $\eps$ and $T$ since $N$ depends on $\eps$ and $T$.
This dependence is explicitly discussed in the last paragraph of the proof.
Then, with the understanding that $\wt{\Omega}_{\eps}^c = \Omega_B \setminus \wt{\Omega}_{\eps}$, we have 
$ d (a_0, b_0) \otimes \rho( \wt{\Omega}_{ \eps}^c) <  Z^{-\frac{1}{2}}{\eps^2} $.
Then, by Lemma \ref{LEM:GABSCONTI} and Cauchy-Schwarz inequality, 
we have $\mu( \wt{\Omega}_{ \eps}^c) < {\eps}$.

Let $(u_0, v_0) \in \wt{\Omega}_{ \eps}$.
This implies that for $(U_0, V_0) = (u_0^N, v_0^N) \in \Omega_{N, \eps}$.
Then, as in the proof of Lemma \ref{finitedimgrowth},  we have 
$ 
\| (U, V ) ( j \dl) \|_{H^{s_1, s_2} } \leq K 
$ 
for  $j = 0, 1 , \cdots, \big[\frac{T}{\dl} \big]$.   
Also, from the local theory, we have
$\| (u, v ) (t) \|_{H^{s_1, s_2} }$,  
$\| (U, V ) (t) \|_{H^{s_1, s_2} } \leq 2K$
for $|t| \leq \dl$.

Now, consider the difference of the solutions $(u, v)$ and $(U, V)$ to \eqref{MB}and  \eqref{approxMB}.
By writing as integral equations, we have
\[ \begin{cases} 
u(t) - U(t) = S(t) \big(u_0 - U_0\big) - \int_0^t S(t - t') F(t') d t'\\  
v(t) - V(t) = S_\al(t) \big(v_0 - V_0\big) -\int_0^t  S_\al(t - t') G(t')d t', 
\end{cases}
\]
where 
$F(t) = \dx \big(\frac{v^2}{2}\big)(t) - \mathbb{P}_N \dx \big(\frac{V^2}{2}\big)(t)$ and 
$G(t) = \dx \big(uv\big)(t) - \mathbb{P}_N  \dx \big(UV\big)(t)$.
Now, let $w = (u, v)$ and $W = (U, V)$.
From the linear estimates,  we have 
\begin{equation} \label{Ghomo2}
\big\| \eta(t) \big(S(t), S_\al(t)\big) (w_0 - W_0) \big\|_{Y^{s_1, s_2}} \lesssim \| w_0 - W_0 \|_{H^{s_1, s_2}}.
\end{equation}

\noindent
Since $\mathbb{P}_N \Big( \big( \mathbb{P}_\frac{N}{2}v\big)^2 \Big) =  \big( \mathbb{P}_\frac{N}{2}v\big)^2 $, we have
\begin{align*}
F  = \frac{1}{2} \dx \Big(  v^2 - \big( \mathbb{P}_\frac{N}{2} v\big)^2 \Big)
+ \frac{1}{2} \mathbb{P}_N \dx \Big( \big( \mathbb{P}_\frac{N}{2}v\big)^2 - v^2 \Big) 
+ \frac{1}{2} \mathbb{P}_N\dx (v^2 - V^2).
\end{align*}

\noindent
Then, from the local theory  along with the boundedness of $\mathbb{P}_N$, we have
\begin{align*}
\bigg\|\eta_{_\dl} &(t) \int_0^t S(t - t')  F(t') d t'\bigg\|_{Y^{s_1, s_2}[-\dl, \dl]}
\lesssim \big\|F \big\|_{Z^{s_1, s_2}[-\dl, \dl]} \\
& \lesssim \left\|  \dx \big(v + \mathbb{P}_\frac{N}{2} v \big)\big(v - \mathbb{P}_\frac{N}{2} v \big) \right\|_{Z^{s_1, s_2}[-\dl, \dl]} 
+ \left\|  \dx \big( v +V \big) \big( v -V \big) \right\|_{Z^{s_1, s_2} [-\dl, \dl] } \\
& \lesssim \dl^\theta \Big( \big\|v + \mathbb{P}_\frac{N}{2} v \big\|_{Y_\al^{s_1, s_2}}
\big\|v - \mathbb{P}_\frac{N}{2} v \big\|_{Y_\al^{s_1, s_2}[-\dl, \dl]}   
+ \|v + V \|_{Y_\al^{s_1, s_2}}
\|v - V\|_{Y_\al^{s_1, s_2}[-\dl, \dl]}  \Big)
\end{align*}
for some $\theta > 0$.
Note that 
$\big\|v - \mathbb{P}_\frac{N}{2} v \big\|_{Y_\al^{s_1, s_2}[-\dl, \dl]}  \lesssim N^{-\g} K$
and 
$\big\|\mathbb{P}_\frac{N}{2} v \big\|_{Y_\al^{s_1, s_2}[-\dl, \dl]} $, 
$\| V \|_{Y_\al^{s_1, s_2}[-\dl, \dl]} \leq  2K $.
Hence, we have
\begin{equation} \label{Gduhamel1}
\bigg\|\eta_{_\dl}(t)\int_0^t  S(t - t')    F(t') d t'\bigg\|_{Y^{s_1, s_2}[-\dl, \dl]}  
 \leq C \dl^\theta \big(N^{-\g} + \| w - W  \|_{Y^{s_1, s_2} \times Y_\al^{s_1, s_2}[-\dl, \dl]} \big).
\end{equation}

\noindent
Similarly, we have
\begin{align*}
G & = \dx\big( u - \mathbb{P}_\frac{N}{2}u  \big)v
+ \dx (\mathbb{P}_\frac{N}{2}u) ( v-  \mathbb{P}_\frac{N}{2}v) 
 + \mathbb{P}_N \dx (\mathbb{P}_\frac{N}{2}u) (\mathbb{P}_\frac{N}{2}v -  v ) \\
&+ \mathbb{P}_N \dx ( \mathbb{P}_\frac{N}{2}u - u ) v
 + \mathbb{P}_N \dx u( v - V)
+ \mathbb{P}_N \dx ( u - U)V
\end{align*}

\noindent
and thus 
\begin{equation} \label{Gduhamel2}
 \bigg\|\eta_{_\dl}(t)\int_0^t  S_\al(t - t')   G(t') d t'\bigg\|_{Y_\al^{s_1, s_2}[-\dl, \dl]} 
 \leq C \dl^\theta \big(N^{-\g} + \| w- W  \|_{Y^{s_1, s_2} \times Y_\al^{s_1, s_2}[-\dl, \dl]} \big).
\end{equation}

\noindent 
From  \eqref{Ghomo2},  \eqref{Gduhamel1}, and  \eqref{Gduhamel2}, 
we have
\begin{align*}
\| w - W  \|_{Y^{s_1, s_2}  \times Y_\al^{s_1, s_2}  [-\dl, \dl]}
 & \leq  C \| w_0 - W_0 \|_{H^{s_1, s_2} }  \\
& + C \dl^\theta \big(N^{-\g} + \| w - W  \|_{Y^{s_1, s_2} \times Y_\al^{s_1, s_2}[-\dl, \dl]} \big).
\end{align*}

\noindent
Then, by choosing $\dl$ sufficiently small, it follows from \eqref{Ghomo1} that
\begin{align*} 
\| (w - W) (\dl) \|_{H^{s_1, s_2}} \lesssim \| w - W \|_{Y^{s_1, s_2}  \times Y_\al^{s_1, s_2}  [-\dl, \dl]} 
 \lesssim \| w_0 - W_0 \|_{H^{s_1, s_2} } \lesssim N^{-\g} K. 
\end{align*}

\noindent
By choosing $N$ large such that $ \big[\frac{T}{\dl}\big] KN^{-\g} \ll 1$,
we can repeat this argument $\big[\frac{T}{\dl}\big]$ times over the intervals $[j \dl, (j+1) \dl]$ for $j = 0, 1, \cdots, \big[\frac{T}{\dl}\big]-1$
and obtain
$ \| (u, v)( j \dl) \|_{H^{s_1, s_2}} \leq K + 1$
for $ j = 0, 1, \cdots, \big[\frac{T}{\dl}\big] $. 
Hence, from the local theory and the time-reversibility of the equation,  the solution $(u, v)$  
with the initial data $(u_0, v_0) \in \wt{\Omega}_\eps$ exists on $[-T, T]$ and moreover we have
$ \| (u, v)( t ) \|_{H^{s_1, s_2}} \leq 2(K + 1) \sim \big( \log \frac{T}{\eps} \big)^\frac{1}{2}$
for all $ |t| \leq T$.

Note that  $\wt{\Omega}_\eps$ constructed above depends on $N$, $T$,  and $\eps$, 
where $N$, in turn, depends on $T$ and $\eps$.
To be explicit about this dependence, 
let us denote $\wt{\Omega}_\eps$ and $N$ by $\wt{\Omega}(T, \eps)$ and $N(T, \eps)$.
Now, fix $\eps > 0$, and 
let $T_j = 2^j$ and $\eps_j = \frac{\eps}{2^{j+1}}$ for $j \in \mathbb{N}$. 
Then, construct $\wt{\Omega}^{(j)}_\eps = \wt{\Omega}(T_j, \eps_j)$ with $N_j = N(T_j, \eps_j)$.
By construction, $\mu\big(({\wt{\Omega}^{(j)}_\eps})^c\big) < \eps_j $.
Note that $K_j \sim \big(\log \frac{T_j}{\eps^2_j} \big)^\frac{1}{2}
= \big(\log \frac{2^{3j + 2}}{\eps^2} \big)^\frac{1}{2}
\sim \big(\log \frac{{T_j}}{\eps} \big)^\frac{1}{2}$.
Thus, we can choose $N_j$ sufficiently large so that 
$\big[\frac{T_j}{\dl_j}\big] K_j N_j^{-\g} \lesssim T_j^{1+} N_j^{-\g} \ll 1$.
Also, for $(u_0, v_0) \in \wt{\Omega}^{(j)}_\eps$, we have 
\[\| (u, v)(t)\|_{H^{s_1, s_2}} \lesssim \Big( \log \frac{T_j\cdot 2^{j+1}}{\eps} \Big)^\frac{1}{2}
= \Big( \log \frac{ 2^{2j+1}}{\eps} \Big)^\frac{1}{2}
\sim \Big( \log \frac{ 2^{j}}{\eps} \Big)^\frac{1}{2} 
=\Big( \log \frac{T_j }{\eps} \Big)^\frac{1}{2}\]

\noindent
for $|t| \leq T_j$.
Finally, let $\Omega_\eps = \bigcap_{j = 1}^\infty \wt{\Omega}^{(j)}_\eps$.
Then, $\Omega_\eps$ has the desired property.
\end{proof}

\begin{remark} \rm
This establishes the global well-posedness of the Majda-Biello system
almost surely on the statistical ensemble (with the $L^2$ cutoff and the mean 0 assumption on $u_0$.)
\end{remark}

As a corollary, we obtain
\begin{corollary}  \label{unifconvMB2}
Let  $\frac{1}{4}< s_1 < \s_1 < \frac{1}{2} < s_2 < \s_2 < 1$ with $ 2s_1 > s_2 > \frac{1}{2} + \frac{1}{2} \max( \nu_{c_1},\nu_{d_1},\nu_{d_2})$ 
and $2\s_1 > \s_2$. 
Also, let $\Omega_\eps$ be as in  Proposition \ref{GWPensemble2}.
Then, for $T < \infty$, we have
\[ \big\| (u, v) - (u^N, v^N) \big\|_{C([-T, T] ; H^{s_1, s_2} \times H^{s_1, s_2})} \to 0\]
as $ N \to \infty$
uniformly for $ (u_0, v_0) \in \Omega_\eps$.
\end{corollary}

Now, we are ready to prove the invariance of $\mu$. 
Let \[X = \bigcup_M \big\{ f = f\big( (a_n)_{|n| \leq M}, (b_n)_{|n| \leq M} \big) \text{ continuous and bounded}\big\}.\]

\noindent
i.e. $f \in X$ is bounded and there exists $M$ such that $f$ depends  continuously on a finitely many modes $\{ |n| \leq M\}$. 
Let $\cj{X}$ be the closure of $X$.
Also, for $|t| < \infty $, let $S^t$ be the flow map for the Majda-Biello system \eqref{MB} 
and $S_N^t$ be the flow map for its finite dimensional approximation \eqref{approxMB}.

Note that $\rho_N$ is obtained from $\rho$ by integrating in $(a_n, b_n)_{n > N}$. 
From Sobolev inequality, 
\[\bigg|\int \mathbb{P}_N\phi (\mathbb{P}_N \psi)^2 - \phi \psi^2 dx\bigg|
\lesssim \|\phi^N - \phi\|_{H^\frac{1}{6}} \|\psi\|_{H^\frac{1}{6}}^2
+\|\phi\|_{H^\frac{1}{6}} \|\psi\|_{H^\frac{1}{6}}\|\psi^N - \psi\|_{H^\frac{1}{6}}
\to 0, \text{ a.s.}
\] 

\noindent
 as $N \to \infty$ since $(\phi, \psi) \in H^\frac{1}{6}$ a.s.
Also, 
we have 
$e^{\frac{1}{2} \int (\mathbb{P}_N\phi- a_0) (\mathbb{P}_N\psi - b_0)^2} \chi_{\Omega_{N, B}} \leq 
e^{c_1 \|\phi - a_0\|_{H^\frac{1}{6}} + c_2 \|\psi - b_0\|_{H^\frac{1}{6}} }\chi_{\Omega_B} \in L^1(d\rho)$
for any $N$.
Thus, by Dominated Convergence Theorem, we have
$\lim_{N\to \infty }Z_N = Z$.
Then, from the boundedness of $f$ and Dominated Convergence Theorem,  we see that 
\begin{align} \label{Gdecay}
& I_N := \int f\big(S^t_N(\mathbb{P}_N \phi, \mathbb{P}_N \psi)\big) (d\mu - d\mu_N)\\
& = \int \Big\{ Z_N^{-1}e^{\frac{1}{2} \int \mathbb{P}_N\phi (\mathbb{P}_N\psi)^2} \chi_{\Omega_{N, B}} 
- Z^{-1}  e^{\frac{1}{2} \int \phi \psi^2 } \chi_{\Omega_{B}} \Big\} 
f\big(S^t_N(\mathbb{P}_N \phi, \mathbb{P}_N \psi)\big) d (a_0, b_0) \otimes d\rho\notag
\end{align}

\noindent
tends to 0 as $N \to \infty$.

\begin{theorem} \label{THM:GIBBSinvariance}
The Gibbs measure $\mu$ is invariant under the flow of the Majda-Biello system \eqref{MB} in the sense that
  \[ \int f\big( S^t(\phi, \psi)\big) \, \mu(d\phi, d\psi) = \int f(\phi, \psi) \, \mu(d\phi, d\psi)\]
for all $f \in \cj{X}$.
\end{theorem}

\begin{proof}
Fix $f \in X$, $t > 0$, and $\eps > 0$.
Let   
$ \wt{I}_N(A) =  \big| \int_{ A}  f(  S^t\varphi )  - f( S_N^t \mathbb{P}_N\varphi) d \mu \big| $,
where $\varphi = (\phi, \psi)$.
By Corollary \ref{unifconvMB2}, we have
$ \| S^t\varphi - S^t_N \mathbb{P}_N \varphi \|_{ H^{s_1, s_2}} \to 0$ as $ N \to \infty$
uniformly for $(\phi, \psi) \in \Omega_\eps$.
Since $f$ is continuous, there exists $N_1$ such that
\[  \wt{I}_N( \Omega_\eps)
 \leq \sup_{ \varphi \in \Omega_\eps} | f( S^t\varphi) - f( S_N^t\mathbb{P}_N \varphi) | 
 < \eps , \]

\noindent
for all $N \geq N_1$. 
On $\Omega_\eps^c$, we have
$\wt{I}_N( \Omega^c_\eps)
\leq 2 \|f\|_{L^\infty} \mu(\Omega^c_\eps) \lesssim \eps.$
From \eqref{Gdecay}, there exists $N_2$ such that $|I_N| < \eps$
 for all $ N \geq N_2$. 
Putting all together, we have
\begin{align} \label{Gibbsest}
  \bigg|  \int f & (  S^t\varphi) d\mu  - \int f( S_N^t\mathbb{P}_N\varphi) d \mu_N \bigg| 
\leq \wt{I}_N(\Omega_\eps) + \wt{I}_N  (\Omega_\eps^c)
 + |I_N|
\lesssim \eps 
\end{align}

\noindent
 for all $ N \geq \max ( N_1, N_2)$.
Therefore, from \eqref{Gibbsest} and the invariance of $\mu_N$, we have
\begin{align} \label{invariance2}
\int f(  S^t \varphi) d \mu & = \lim_{N \to \infty} \int f( S_N^t \mathbb{P}_N\varphi) d \mu_N 
 = \lim_{N\to \infty} \int f(\mathbb{P}_N\varphi) d \mu_N
= \int f (\varphi) d \mu. 
\end{align}

\noindent
By density, \eqref{invariance2} holds for all $f \in \cj{X}.$
\end{proof}

\section{Appendix}
In this appendix, we present the proof of 
Lemma \ref{LEM:tight} following the notations introduced in Section 3.
We only show the second estimate in \eqref{tight1}.
From  H\"older inequality and Lemma \ref{LEM:GABSCONTI}, we have
\begin{align*}
\mu & \big(\Omega_{ B}  \setminus \Omega_{ B} (s_1, s_2, K) \big)
= \int_{ \{ \| (\phi, \psi) \|_{H^{s_1, s_2}} > K \} }  \chi_{\Omega_{ B}} d \mu \\
& \lesssim B^2 \bigg(\int_{ \{ \| (\phi, \psi) \|_{H^{s_1, s_2}} > K \} }  \chi_{\Omega_{ B}} d \rho \bigg)^\frac{1}{2} 
 \Big\| \chi_{ \left\{ \| (\phi, \psi) \|_{H^{s_1, s_2}} > K \right\} }  
 \chi_{\Omega_{ B}}  e^{\frac{1}{2} \int \phi \psi^2 dx} 
\Big\|_{L^2 ( d (a_0, b_0) \otimes d \rho)} \\
& \lesssim \rho\big( \| (\phi, \psi) \|_{H^{s_1, s_2}} > K, \  \|(\phi, \psi) \|_{L^2} \leq B, \phi, \psi \text{ mean } 0 \big)^\frac{1}{2}.
\end{align*}

\noindent
For notational simplicity, we assume $\phi$ and $\psi$ have mean 0 in the following.
By the definition of the $H^{s_1, s_2}$ norm, we have
\begin{align} \label{XYZ}
\rho\big( \| (\phi, \psi) \|_{H^{s_1, s_2}}  > K, \  \|(\phi, \psi) \|_{L^2} \leq B \big) 
 \leq \rho\big( \| (\phi, \psi) \|_{H^{s_1}} > \tfrac{1}{2} K, \  \|(\phi, \psi) \|_{L^2} \leq B \big) \notag \\
 + \rho\big( \sup_n \jb{n}^{s_2} | \ft{\phi}(n)| + \sup_n \jb{n}^{s_2} | \ft{\psi}(n)|> \tfrac{1}{2} K, \  \|(\phi, \psi) \|_{L^2} \leq B \big).
\end{align}

Let $d \rho_1= d\rho\big|_{(a_n)_{n \geq 1}}$ and $d\rho_2 = d\rho\big|_{(b_n)_{n \geq 1}}$.
Then, we have
\begin{align*}
\rho\big( \| (\phi, \psi) \|_{H^{s_1}} > \tfrac{1}{2} K, \  \|(\phi, \psi) \|_{L^2} \leq B \big)
 \leq \ & \rho_1 \big( \{\| \phi \|_{H^s} \geq \tfrac{1}{4}K, \|\phi\|_{L^2} \leq B \}\big) \\ 
 + & \rho_2 \big( \{\| \psi \|_{H^s} \geq \tfrac{1}{4} K, \| \psi\|_{L^2} \leq B \}\big). 
\end{align*}

\noindent
We only prove the estimate on $\rho_1$.
Fix $M_0$ dyadic (to be specified later.) Then,  we have
$\big\|\sum_{|n| \leq M_0} \ft{\phi}(n) e^{inx}  \big\|_{H^s} 
\leq C  M_0^s \big( \sum_{|n| \leq M_0} |\ft{\phi}(n)|^2\big)^\frac{1}{2} \leq C  M_0^s B.$
By choosing $C M_0^s B = \frac{1}{8} K$, we have
$\big\|\sum_{|n| \leq M_0} \ft{\phi}(n) e^{inx}  \big\|_{H^s} \leq \frac{1}{8}K$
and $ M_0 \sim \big(\frac{K}{B}\big)^\frac{1}{s}. $

Now, let $\s_j = C 2^{-\eps j}$ for some small $\eps > 0$,
where $C$ is chosen such that $\sum_{j \geq 1} \s_j = \frac{1}{8}$.
Also, let $M_j = M_0 2^j$ dyadic.
Note that $\s_j = C M_0^\eps M_j^{-\eps}$.
Then,  we have 
\[\rho_1\big( \| \phi\|_{H^s} > \tfrac{1}{4}K, \  \| \phi\|_{L^2} \leq B \big)
\leq \sum_{j = 1}^\infty \rho_1\big[ \Big\| \sum_{|n| \sim M_j} \ft{\phi}(n) e^{inx} \Big\|_{H^s} > \s_j K \big].\]

\noindent
Recall that  $\ft{\phi}(n) = \frac{f_n(\omega)}{n}$, 
where $\{f_n(\omega)\}_{n\geq 1}$ are i.i.d. standard complex Gaussian random variables 
and $f_{-n} = \cj{f_n}$. 
Thus, if $\big\| \sum_{|n| \sim M_j} \ft{\phi}(n) e^{inx} \big\|_{H^s}  \geq  \s_j K$, then we have
 $\big( \sum_{n\sim M_j} |f_n(\omega)|^2\big)^\frac{1}{2} \gtrsim R_j := \s_j KM_j^{1-s} $.
Then, using the polar coordinates,  we have
\begin{align} \label{GLR6}
\mathbb{P}_\omega \big[   \Big( \sum_{n\sim M_j} |f_n(\omega)|^2\Big)^\frac{1}{2} \gtrsim R_j \big]
 \sim \int_{B^c(0, R_j)} e^{-\frac{|f|^2}{2}} \prod_{n \sim M_j} df_n
 \lesssim \int_{R_j}^\infty e^{-\frac{r^2}{2}} r^{2\cdot \# \{n \sim M_j\} -1} dr.
\end{align}

\noindent 
Note that the implicit constant in the inequality is $\s(S^{2\cdot \# \{n \sim M_j\} -1})$,
a surface measure of the $2\cdot \# \{n \sim M_j\} -1$ dimensional unit sphere.
We drop it since $\s(S^n) = 2\pi^\frac{n}{2} /\G(\frac{n}{2}) \lesssim 1$.
By change of variables $t = M_j^{-\frac{1}{2}} r$, we have
$r^{2\cdot \# \{n \sim M_j\} -2} \lesssim r^{4M_j} \sim M_j^{2M_j} t^{4M_j}.$
Since $t \geq M_j^{-\frac{1}{2}} R_j = C K M_0^\eps M_j^{\frac{1}{2} -s-\eps} \gtrsim K M_0^\eps M_j^{0+}$
as long as $s < \frac{1}{2}$ (with $\eps>0$ sufficiently small), we have
\begin{equation}\label{GLR100}
M_j^{2M_j} = e^{2M_j \ln M_j} < e^{\frac{1}{8}M_jt^2}, \ 
\text{ and } \ 
 t^{4M_j}  = (t^4) ^M_j < (e^{\frac{1}{8}t^2})^{M_j} = e^{\frac{1}{8}M_jt^2}
\end{equation}

\noindent
for $K$ sufficiently large, independent of $M_j > M_0$.
Thus, we have
$r^{2\cdot\# \{n \sim M_j\} -2} < e^{\frac{1}{4}M_j t^2} = e^{\frac{1}{4}r^2}$
for $r > R$.
From this and \eqref{GLR6}, we have
\begin{equation*}
\mathbb{P}_\omega\big[   \Big( \sum_{n\sim M_j} |f_n(\omega)|^2\Big)^\frac{1}{2} \gtrsim R_j \big]
\leq C \int_{R_j}^\infty e^{-\frac{1}{4}r^2} r dr = C e^{-\frac{1}{4}R_j^2} \leq e^{-cR_j^2} 
= e^{-c\s_j^2 K^2 M_j^{2-2s}}.
\end{equation*}

\noindent
Then, by summing up over $j$, 
 we obtain
\[ \rho_1\big(  \| \phi\|_{H^s} >\tfrac{1}{4}K, \  \| \phi\|_{L^2} \leq B \big)
\leq \sum_{j = 1}^\infty e^{-c\s_j^2 K^2 M_j^{2-2s}}  \leq e^{-c' K^2M_0^{2-2s}} \leq   e^{-c''K^2}. \]

As for the second term in \eqref{XYZ}, we have
\begin{align*}
 \rho\big( \sup_n \jb{n}^{s_2} &  | \ft{\phi}(n)|  + \sup_n \jb{n}^{s_2} | \ft{\psi}(n)|> \tfrac{1}{2} K, \  \|(\phi, \psi) \|_{L^2} \leq B \big) \\
 & \leq \rho_1\big( \sup_n \jb{n}^{s_2} | \ft{\phi}(n)|  > \tfrac{1}{4} K \big) 
+ \rho_2\big( \sup_n \jb{n}^{s_2} | \ft{\psi}(n)|> \tfrac{1}{4} K\big). 
\end{align*}

\noi
First, recall the following integrability result due to Fernique \cite{FER}
for an abstract Wiener space $(i, H, B)$.

\begin{proposition} [Theorem 3.1 in \cite{HUO}]
\label{THM:FER} 
Let $(i, H, B)$ be an abstract Wiener space.
Then, there exists $ c > 0$ such that $ \int_B e^{c \|x\|_B^2} \mu(d x) < \infty$.
Hence, there exists $ c' > 0$ such that $\rho ( \|x\|_B > K) \leq e^{-c'K^2}$.
\end{proposition}

\noindent
Now, define a Banach space $B^{s_2}$ via the norm $\|\cdot\|_{B^{s_2}} = \sup_n \jb{n}^{s_2} | \ft{\phi}(n)|$.
Then, \eqref{tight1} follows once we show that $(i, H^1_0, B^{s_2})$ is an abstract Wiener space.

First, recall that $\jb{n}^{s_2} |\ft{\phi}(n)| \sim \jb{n}^{s_2-1} |f_n(\omega)|$.
Let $X_n (\omega)= |f_n(\omega)|^{\frac{1}{1-s_2}}$.
Then, we have $ \mathbb{E} [X_n ] < \infty$, 
which immediately implies that $\jb{n}^{-1}X_n \to 0$ a.s.
Then, by Egoroff's Theorem, given $\eps > 0$ there exists a set $E \subset \Omega$
with $\rho(E^c) < \eps$ 
such that 
$\jb{n}^{s_2} |\ft{\phi}(n)| \to 0 $ uniformly on $E$.
Then, choose $N_0$ sufficiently large such that 
$\jb{n}^{s_2} |\ft{\phi}(n)| < \eps $ on $E$ for all $|n| \geq N_0$.
This shows that $\rho( \| \mathbb{P}_{\geq N_0} \phi\|_{B^{s_2}}\geq \eps) < \eps$
as desired.

\end{document}